\documentclass[12pt]{amsart}

\usepackage{comment}
\usepackage[normalem]{ulem}
\usepackage[colorlinks,linkcolor=red,anchorcolor=blue,citecolor=blue]{hyperref} 
\usepackage[abbrev,alphabetic]{amsrefs}
\usepackage{amscd} 
\usepackage{amsmath}
\usepackage{mathtools}
\usepackage{amssymb} 
\usepackage{amsthm}
\usepackage{dsfont}
\usepackage{bigdelim}
\usepackage{color} 
\usepackage{enumerate}
\usepackage{graphicx}
\usepackage{mathrsfs}
\usepackage{multirow}
\usepackage[all]{xy} 
\usepackage{color} 
\usepackage{enumitem}
\usepackage[dvipsnames]{xcolor}

\newtheorem{thm}{Theorem}[section] 

\newtheorem{cor}[thm]{Corollary}
\newtheorem{prop}[thm]{Proposition}
\newtheorem{conj}[thm]{Conjecture}

\newtheorem{lem}[thm]{Lemma}
\theoremstyle{definition} 
\newtheorem{defn}[thm]{Definition}

\theoremstyle{remark}
\newtheorem{rem}[thm]{Remark}
\newtheorem{setup}[thm]{Setup}

\newtheorem{claim}[thm]{Claim}

\numberwithin{equation}{section}
\newtheorem{case}{Case}

\newcommand{\rk}[0]{\operatorname{rk}}

\newcommand{\codim}[0]{\operatorname{codim}}

\newcommand{\pr}{{\rm pr}}

\newcommand{\End}[0]{\operatorname{End}}
\newcommand{\Coker}[0]{\operatorname{Coker}}

\newcommand{\GL}[0]{\operatorname{GL}}

\newcommand{\PGL}[0]{\mathbb{P}\GL(r,\C)}

\newcommand{\Alb}{{\rm Alb}}

\newcommand{\reg}{{\rm{reg}}}

\newcommand{\Aut}[1]{\mathrm{Aut}(#1)}
\newcommand{\Ker}[1]{\mathrm{Ker}(#1)}
\newcommand{\Image}[1]{\mathrm{Im}(#1)}

\newcommand{\underalign}[2]{\quad \underset{\mathclap{\strut #1}}{#2}\quad}
\newcommand{\polar}{\beta}
  
\newcommand{\R}{\mathbb{R}}

\newcommand{\C}{\mathbb{C}}
\newcommand{\Q}{\mathbb{Q}}

\title[Minimal projective varieties satisfying Miyaoka's equality]
{Minimal projective varieties 
\\satisfying Miyaoka's equality}

\author{Masataka IWAI}
\address{Department of Mathematics, Graduate School of Science, the University of Osaka,
1-1, Machikaneyama-cho, Toyonaka, Osaka 560-0043, Japan.}
\email{{\tt masataka@math.sci.osaka-u.ac.jp}}
\email{{\tt masataka.math@gmail.com}}

\author{Shin-ichi MATSUMURA}
\address{Mathematical Institute 
$\&$ Division for the Establishment of Frontier Science of Organization for Advanced Studies, 
Tohoku University, 
6-3, Aramaki Aza-Aoba, Aoba-ku, Sendai 980-8578, Japan.}
\email{{\tt mshinichi-math@tohoku.ac.jp}}
\email{{\tt mshinichi0@gmail.com}}

\author{Niklas M\"uller}
\address{Department of Mathematics, Universit\"at Duisburg-Essen,
Thea-Leymann-Str. 9, 45127 Essen, Germany.}
\email{{\tt niklas.mueller@uni-duisburg-essen.de}}

\date{\today}



\begin{document}

\begin{abstract}
In this paper, we establish a structure theorem for a minimal
projective klt variety $X$ satisfying Miyaoka's equality $3c_2(X) = c_1(X)^2$.
Specifically, we prove that the canonical divisor $K_X$ is semi-ample
and that the Kodaira dimension $\kappa(K_X)$ is equal to $0$, $1$, or $2$.
Furthermore, based on this abundance result,
we show that a maximally quasi-\'etale cover of $X$ is smooth,
and we explicitly describe the structure of the Iitaka fibration.
In addition, we prove an analogous result for projective klt varieties
with nef anti-canonical divisor.
\end{abstract}

\maketitle

\tableofcontents

\section{Introduction}

\subsection{Main results}

The abundance conjecture, one of the most significant problems in algebraic geometry, 
remains largely unsolved in higher dimensions.

\begin{conj}[Abundance conjecture for minimal klt varieties]
\label{conj-abundance-conjecture}
Let $X$ be a $($complex$)$ projective klt variety. 
If $X$ is minimal $($i.e.,\,if the canonical divisor $K_X$ is nef$)$, 
then $K_X$ is semi-ample.
\end{conj}

In this paper, we establish a structure theorem  
for minimal projective klt varieties whose Chern classes satisfy a certain extremal condition, the so-called \textit{Miyaoka equality}. As part of our results, we also resolve the abundance conjecture for such varieties.

Building on earlier work \cite{Miy77, Yau77}, Miyaoka proved in \cite[Chapter~7]{Miy87} that a smooth minimal projective variety $X$ satisfies \textit{Miyaoka's inequality}: 
$$
3c_{2}(\Omega_{X}^{1}) - c_{1}(\Omega_{X}^{1})^{2} \geq 0.
$$ 
Subsequently, numerous analogous inequalities have been established; see, for example, \cite[Theorem~5.6]{Lan02}, \cite[Theorem~7.2]{RT16}, \cite{LT18, GKPT19a, RT22, GT22, LL23}, and the references therein.
These inequalities involving Chern classes have a long history and play an important role in higher-dimensional algebraic geometry.

The structure of varieties $X$ satisfying equality in such inequalities is expected to be well-understood in detail.
For instance, building on \cite{Yau77}, Miyaoka proved in \cite{Miy84} that a smooth projective surface $X$ of general type satisfying Miyaoka's equality
$$
3c_{2}(\Omega_{X}^{1}) - c_{1}(\Omega_{X}^{1})^{2} = 0
$$
must be an (infinite) \'etale quotient of the unit ball. Later, Peternell-Wilson \cite{PW96} explicitly determined which minimal terminal varieties of dimension three can satisfy Miyaoka's equality. 
The case $\nu(K_X) = 0$, where Miyaoka's equality is equivalent to $c_2(\Omega_X^1) = 0$ and $X$ must be a finite \'etale quotient of a complex torus, has been extensively studied, including in the singular case (see \cite{GKP16, LT18, CGG23}).
More recently, the case where $\nu(K_X) \leq 1$ and $c_2(\Omega_X^1) = 0$ has been studied for compact K\"ahler manifolds in \cite{IM22}, where it was shown that $X$ admits the structure of a smooth abelian fibration over a curve of general type up to finite \'etale cover. We emphasize that in all of the above cases, the abundance conjecture was shown to hold for $X$. Moreover, varieties satisfying related extremal Chern class conditions have also been studied in \cite{GKPT19a, HS21, GKP21, Pat23}.

Our main result extends, generalizes, and unifies several of the aforementioned results \cite{Miy84, Miy87, PW96, IM22}.  
Specifically, we first prove that Miyaoka's inequality continues to hold for minimal projective \textit{klt} varieties.  
Moreover, we resolve the abundance conjecture for a minimal klt variety $X$ satisfying the equality
$$
3c_{2}(\Omega_{X}^{1}) - c_{1}(\Omega_{X}^{1})^{2} = 0,
$$
and explicitly describe the geometry of $X$.  
In particular, for minimal terminal threefolds, Theorem~\ref{thm-main1} precisely recovers the classification given in \cite{PW96}.

\begin{thm}[{Main Result}]\label{thm-main1}
Let $X$ be a projective klt variety of dimension $n$ with nef canonical divisor $K_{X}$. Then, the following statements hold$:$

\smallskip

$(A)$ Miyaoka's inequality holds for ample divisors $H_1, \ldots, H_{n-2}$ on $X$$:$
\begin{equation*}
    \left( 3 \widehat{c}_2(\Omega_{X}^{[1]}) - \widehat{c}_1(\Omega_{X}^{[1]})^2 \right)
H_1\cdots H_{n-2} \geq 0, 
\end{equation*}
where $\widehat{c}_2(\Omega_{X}^{[1]})$ and 
$\widehat{c}_1(\Omega_{X}^{[1]})$ denote the $\Q$-Chern classes 
of the cotangent sheaf $\Omega_{X}^{[1]}$. 

\smallskip

$(B)$ Assume that Miyaoka's equality holds for some ample divisors $H_i$ on $X$$:$
\begin{equation}
\label{eq-miyaoka}
 \left( 3 \widehat{c}_2(\Omega_{X}^{[1]}) - \widehat{c}_1(\Omega_{X}^{[1]})^2 \right)
H_1\cdots H_{n-2}=0.
\end{equation}
Then, the canonical divisor $K_X$ is semi-ample and $\nu(K_{X})=\kappa(K_{X})$ is equal to $0$, $1$, or $2$. 

\smallskip

Moreover, there exists a finite quasi-\'etale cover $X'\rightarrow X$ such that, depending on the Kodaira dimension, one of the following holds:
\begin{itemize}
\item[$(i)$] In the case where $\nu(K_{X})=\kappa(K_{X})=0$, 
the variety $X'$ is isomorphic to an abelian variety. 

\item[$(ii)$] In the case where $\nu(K_{X})=\kappa(K_{X})=1$, 
the variety $X'$ admits the structure of an abelian group scheme $X' \rightarrow C$  over a curve $C$ of general type.

\item[$(iii)$] In the case where $\nu(K_{X})=\kappa(K_{X})=2$, the variety $X'$ is isomorphic to the product $A \times S$ of an abelian variety $A$ 
and a smooth surface $S$ whose universal cover is the unit ball in $\mathbb{C}^{2}$. 
\end{itemize}
In particular, in all cases, the variety $X$ is smooth up to finite quasi-\'etale covers.
\end{thm}

Note that a similar result was obtained for varieties of large Kodaira dimension in \cite{HS21}, 
where Hao and Schreieder showed that a minimal projective variety $X$ satisfying 
$\widehat{c}_2(\Omega_{X}^{[1]}) K_X^{n-2} = 0$ and $\kappa(K_{X}) = n -1$ 
is isomorphic to the product of an elliptic curve and a variety of general type, 
up to quasi-\'etale covers and birational modifications.  

The following corollary, established for smooth varieties in \cite{IM22}, 
is an immediate consequence of Theorem~\ref{thm-main1}.  
As explained in the introduction of \cite{IM22}, 
this result can be viewed as a partial complement to that of \cite{LP18}.

\begin{cor}\label{cor-Abundance-theorem-c2}
Let $X$ be a projective klt variety with nef canonical divisor. 
Assume that there exist  ample divisors $H_1, \ldots, H_{n-2}$ on $X$ such that
$$
\widehat{c}_2(\Omega_{X}^{[1]}) H_1 \cdots H_{n-2}  = 0.
$$ 
Then, the canonical divisor $K_X$ is semi-ample and $\nu(K_{X})=\kappa(K_{X})$ is equal to $0$ or $1$. 
\end{cor}

In another direction, we obtain a structure theorem (see Theorem~\ref{thm-main2} below) 
for a projective variety $X$ with nef anti-canonical divisor $-K_{X}$ and $\widehat{c}_{2}(\mathcal{T}_{X}) = 0$.  
This theorem generalizes the results of \cite{Cao13, Ou17} for smooth varieties and \cite{IJL23} for terminal threefolds to projective klt varieties of arbitrary dimension.

\begin{thm}\label{thm-main2}
Let $X$ be a projective klt variety of dimension $n$ with nef anti-canonical divisor $-K_{X}$.
Assume that there exist  ample divisors $H_1, \ldots, H_{n-2}$ on $X$ such that 
$$
\widehat{c}_2(\mathcal{T}_{X})H_1\cdots H_{n-2}=0,
$$
where $\mathcal{T}_{X}$ denotes the tangent sheaf of $X$. 
Then, the numerical dimension $\nu(-K_X)$ is equal to $0$ or $1$. Moreover,
there exists a finite quasi-\'etale cover $X'\rightarrow X$ such that, depending on the numerical dimension $\nu(-K_{X})$, one of the following holds$:$
\begin{itemize}
\item[$(i)$] In the case where $\nu(-K_{X})=0$, the variety $X'$ is isomorphic to an abelian variety.

\item[$(ii)$] In the case where $\nu(-K_{X})=1$, the variety $X'$ admits a locally trivial fibration $X' \to A$ onto an abelian variety $A$ with fibre $\mathbb{P}^{1}$.
\end{itemize}
In particular, in all cases, the variety $X$ is smooth up to finite quasi-\'etale covers.
\end{thm}

\subsection{Overview and outline of proofs}

This subsection outlines the proofs of Theorems~\ref{thm-main1} and \ref{thm-main2}, 
and highlights our new contributions compared to previously known results in the literature.  

To prove Miyaoka's inequality for klt varieties, 
we extend the argument of \cite{Lan02} to the klt setting by using Langer's proof of the Bogomolov-Gieseker inequality.  
Note that the argument in \cite{Lan02} is essentially the same as Miyaoka's original proof in \cite{Miy87}, with a slight simplification introduced in \cite{Sim88}. In any case, the possibility of extending these arguments to klt varieties 
might be already  known to experts.

Let us now consider the case where a minimal projective variety $X$ satisfies Miyaoka's equality.  
In the case $\nu(K_{X})=0$, 
the desired result was obtained in \cite[Theorem~1.17]{GKP16} and \cite[Theorem~1.2]{LT18}.  
Thus, it remains to consider two cases: $\nu(K_X) = 1$, discussed in Subsection~\ref{Ssec-proof}, and $\nu(K_X) \geq 2$, discussed in Subsection~\ref{Ssec-proof2}.  

Both cases require a detailed analysis of the Harder-Narasimhan filtration of the cotangent 
sheaf $\Omega_{X}^{[1]}$ and of the second $\mathbb{Q}$-Chern class $\widehat{c}_2(\Omega_{X}^{[1]})$.  
In particular, as in previous works \cite{Miy87, Cao13, Ou17, IM22}, 
the first graded piece $\mathcal{E}_1 \subset \Omega_{X}^{[1]}$ of the Harder-Narasimhan filtration 
plays a crucial role in our proof.  
We study the Harder-Narasimhan filtration in detail in Section~\ref{section:3}.

In the case $\nu(K_{X})=1$, we first prove that $K_{X}$ is semi-ample.  
This part is analogous to the argument in \cite{IM22}, where the result was established in the smooth case. 
The proof relies on Campana's notion of special varieties 
and  the result of Lazi\'c-Peternell \cite{LP18} 
on the abundance conjecture for varieties of numerical dimension one.
At this point, however, our argument departs significantly from the existing literature, 
since it is not clear from the study of the Harder-Narasimhan filtration 
whether a maximally quasi-\'etale cover of $X$ is necessarily smooth.  
To proceed, we first show that the Iitaka fibration $f\colon X \rightarrow C$
 carries the structure of a stratified $C^\infty$-orbifold fibre bundle.  
This requires a precise local understanding of numerically flat sheaves on klt varieties (see Subsection~\ref{subsec-flat}).  
Then, using a combination of topological and algebraic methods, we prove that the multiple fibres can be eliminated.  
Once smoothness is established, we conclude by applying \cite{Hor13, IM22}.

In the case $\nu(K_X) \geq 2$, an examination of the Harder-Narasimhan filtration together with some standard arguments shows that $X$ is smooth.  
The result then follows from an integrability theorem for numerically flat foliations obtained by Pereira-Touzet \cite{PT13}.  
We find it remarkable that, to the best of our knowledge, the conclusion of Theorem~\ref{thm-main1} (iii) had not even been conjectured before, although the subject has been studied extensively in the literature.

The proof of Theorem~\ref{thm-main2} is presented in Section~\ref{Sec-proof3}. 
In the proof, we apply the structure theorem for a projective variety with nef anti-canonical divisor \cite{CH19, MW21}, 
which provides a locally constant Albanese map $X \to \Alb(X)$ 
after replacing $X$ by a finite quasi-\'etale cover. 
We then analyze in detail the Harder-Narasimhan filtration of the relative tangent sheaf $\mathcal{T}_{X/A}$ 
to show that the fibre dimension of $X \to \Alb(X)$ is at most one.

\section*{Acknowledgments}
M.\,I.\ and N.\,M.\ express their gratitude to the organizers of the ``Alpine Meeting on Nonpositive Curvature in K\"ahler Geometry," held in June 2023, where discussions served as a foundation for this paper.  
M.\,I.\ would like to thank Professors Osamu Fujino and Fr\'ed\'eric Touzet for answering his questions, and Xiaojun Wu for pointing out that Corollary~\ref{cor-Abundance-theorem-c2} follows immediately from Theorem~\ref{thm-main1}.  
He would also like to thank Professor Haidong Liu for informing him that Theorem~\ref{thm-main2} follows from \cite{IJL23} for terminal varieties of dimension $3$ and for his further valuable suggestions.  
N.\,M.\ is indebted to Professor Daniel Greb for his valuable comments and suggestions.  

M.\,I.\ was supported by the Grant-in-Aid for Early Career Scientists, No.~22K13907.  
S.\,M.\ was supported by the Grant-in-Aid for Scientific Research (B), No.~21H00976; the Fostering Joint International Research (A), No.~19KK0342; and the JST FOREST Program, No.~JPMJFR2368.  
N.\,M.\ was supported by the DFG Research Training Group 2553, ``Symmetries and Classifying Spaces: Analytic, Arithmetic and Derived."

\section{Preliminary results}\label{sec-pre}

\subsection{Notation and Conventions}\label{subsec-notation}

Throughout this paper, we work over the field of complex numbers.  
We employ the standard notation and conventions of \cite{Har77, KM98}, as detailed in \cite[Subsection~2.1]{IMZ23}.  
Moreover, we will simply say that $X$ is a \emph{klt variety} if the pair $(X,0)$ is klt, 
equivalently, the variety $X$ is log terminal.  
All sheaves considered in this paper are assumed to be coherent, unless explicitly stated otherwise.  
Furthermore, a reflexive sheaf $\mathcal{L}$ of rank one is called a \textit{$\mathbb{Q}$-line bundle} 
if its reflexive $m$-th tensor power 
$\mathcal{L}^{[\otimes m]} := (\mathcal{L}^{\otimes m})^{\vee \vee}$ is locally free for some $m \in \mathbb{N}$.

\subsection{Maximally quasi-\'etale covers}\label{subsec-qec}

In this subsection, following \cite[Section~2.5]{GKP21}, we review maximally quasi-\'etale covers.  

A normal variety $X$ is said to be \textit{maximally quasi-\'etale} if the morphism 
$\widehat{\pi}_{1}(X_{\reg}) \to \widehat{\pi}_{1}(X)$ induced by the natural inclusion 
$i\colon X_{\reg} \hookrightarrow X$ is an isomorphism, where $X_{\reg}$ denotes the smooth locus of $X$ and $\widehat{\pi}_{1}(\bullet)$ denotes the \'etale fundamental group.  
By \cite[Theorem~1.14]{GKP16b}, any klt variety $Y$ admits a finite quasi-\'etale cover $\nu \colon X \rightarrow Y$ such that $X$ is maximally quasi-\'etale.  
If $X$ is maximally quasi-\'etale, then any projective (or linear) representation of the (topological) fundamental group $\pi_{1}(X_{\reg})$ can be extended to a representation of $\pi_{1}(X)$ by the following result:

\begin{thm}[{\cite[Proof of Proposition 3.10]{GKP22}}]
\label{thm-representation-quasi-etale}
Let $X$ be a maximally quasi-\'etale variety.
\begin{itemize}
\item[$(1)$] Any representation $\rho_0 \colon  \pi_{1}(X_{\reg}) \rightarrow \PGL$ factors through $\pi_{1}(X)$$:$
\begin{equation*}
\xymatrix@C=40pt@R=30pt{
\pi_{1}(X_{\reg}) \ar@/^18pt/[rr]^{\rho_0} \ar@{->>}[r]_{i_{*}} & \pi_1(X) \ar@{->}[r]_{\rho\quad} & \PGL.\\ 
 }
 \end{equation*}
\item[$(2)$] 
Any representation $\rho_0 \colon  \pi_{1}(X_{\reg}) \rightarrow \GL(r, \mathbb{C})$ factors through $\pi_{1}(X)$.
\end{itemize}
\end{thm}

\subsection{\texorpdfstring{$\Q$}{Q}-Chern classes}\label{Subsec-Chern}

In this subsection, following \cite{K++, GKPT19b}, we review $\mathbb{Q}$-Chern classes, which allow us to define intersection numbers for characteristic classes of reflexive sheaves.  

Let $X$ be a projective klt variety of dimension $n$, and let $\mathcal{E}$ and $\mathcal{F}$ be reflexive sheaves on $X$.  
As explained in \cite[Chapter~10]{K++} and \cite[Theorem~3.13]{GKPT19b}, the $\Q$-Chern classes 
$\widehat{c}_1(\mathcal{E})$, $\widehat{c}_1(\mathcal{E})\widehat{c}_1(\mathcal{F})$, and $\widehat{c}_2(\mathcal{E})$ 
can be defined as the symmetric $\Q$-multilinear forms satisfying the properties (P1) and (P2):
\begin{align*}
\widehat{c}_1(\mathcal{E})& \colon {\rm N}^{1}(X)_{\Q}^{\,n-1} \longrightarrow \Q, 
\quad (\alpha_1, \ldots, \alpha_{n-1}) \longmapsto \widehat{c}_1(\mathcal{E}) \alpha_1 \cdots \alpha_{n-1}, \\
\widehat{c}_1(\mathcal{E}) \widehat{c}_1(\mathcal{F})& \colon {\rm N}^{1}(X)_{\Q}^{\,n-2} \longrightarrow \Q, 
\quad (\alpha_1, \ldots, \alpha_{n-2}) \longmapsto \widehat{c}_1(\mathcal{E}) \widehat{c}_1(\mathcal{F}) \alpha_1 \cdots \alpha_{n-2}, \\
\widehat{c}_2(\mathcal{E})& \colon {\rm N}^{1}(X)_{\Q}^{\,n-2} \longrightarrow \Q, 
\quad (\alpha_1, \ldots, \alpha_{n-2}) \longmapsto \widehat{c}_2(\mathcal{E}) \alpha_1 \cdots \alpha_{n-2}. 
\end{align*}

\begin{itemize}
\item[(P1)] In the case $n=2$, the surface $X$ admits a (not necessarily quasi-\'etale) finite Galois cover $\nu \colon \widehat{X} \rightarrow X$ such that the reflexive pull-back $\nu^{[*]}\mathcal{E}$ is locally free, and 
$$
\deg \nu \cdot \left( \widehat{c}_{1}(\mathcal{E}) \cdot \alpha \right) 
= c_{1}\left(\nu^{[*]}\mathcal{E}\right) \cdot \nu^{*}\alpha,
$$
for any $\alpha \in {\rm N}^{1}(X)_{\Q}$.

\item[(P2)] In the case  $n>2$, for any general member $V \in \mathcal{B}$ of a free sub-linear system $\mathcal{B} \subset |L|$ of some line bundle $L$, we have
$$
\widehat{c}_1(\mathcal{E}) c_{1}(L) \alpha_2 \cdots \alpha_{n-1}
=
\widehat{c}_1(\mathcal{E}|_{V}) \alpha_2 \cdots \alpha_{n-1},
$$
for any $\alpha_i \in {\rm N}^{1}(X)_{\Q}$.  
Note that $V$ is a klt hypersurface in $X$ and $\mathcal{E}|_{V}$ is reflexive.
\end{itemize}
The same properties as (P1) and (P2) hold for $\widehat{c}_1(\mathcal{E}) \widehat{c}_1(\mathcal{F})$ and $\widehat{c}_2(\mathcal{E})$.  

In this paper, we need to treat intersection numbers with particular care.  
For this reason, we briefly review the definitions of intersection numbers for Weil divisors and torsion-free sheaves.

\begin{defn}

(1) (Weil divisors, divisorial sheaves, and determinant sheaves).  
Let $D$ be a Weil divisor on $X$, and let $\mathcal{F}:=\mathcal{O}_{X}(D)$ be the associated divisorial sheaf.  
We define $\widehat{c}_1(D):=\widehat{c}_1(\mathcal{F})$.  
By properties (P1) and (P2), when $D$ is $\mathbb{Q}$-Cartier, this definition coincides with the $\Q$-multilinear form 
$\frac{1}{m}c_{1}(mD)$ naturally defined by the line bundle $mD$, where $m \in \mathbb{Z}_{+}$ with $mD$ Cartier.  
Furthermore, for a reflexive sheaf $\mathcal{E}$, we have 
$\widehat{c}_1(\mathcal{E})=\widehat{c}_1(\det \mathcal{E})$, 
where $\det \mathcal{E}:=(\Lambda^{\rk \mathcal{E}} \mathcal{E})^{\vee \vee}$.  

\smallskip

(2) (Torsion-free sheaves).  
For a torsion-free sheaf $\mathcal{G}$ on $X$, we define $c_1(\mathcal{G}):=\widehat{c}_1(\mathcal{G}^{\vee \vee})$.  
Let $C$ be the curve defined by the complete intersection of general members of $|H_i|$, where each $H_i$ is a very ample Cartier divisor on $X$.  
Since the natural map $\mathcal{G} \to \mathcal{G}^{\vee \vee}$ is an isomorphism in codimension one, the restriction $\mathcal{G}|_{C}$ is locally free on $C$, and we have
$$
c_{1}(\mathcal{G})H_{1}\cdots H_{n-1} = c_{1}(\mathcal{G}|_{C}).
$$
\end{defn}
For convenience, we often write $c_1(\bullet)$ instead of $\widehat{c}_1(\bullet)$, and similarly 
$c_1(\bullet)c_1(\blacklozenge)$ instead of $\widehat{c}_1(\bullet)\widehat{c}_1(\blacklozenge)$.  
Furthermore, we use the following notation 
$$
D\alpha_1\cdots\alpha_{n-1} := \widehat{c}_1(D)\alpha_1\cdots\alpha_{n-1}.
$$
These conventions are adopted to simplify notation and to avoid confusion in view of the properties mentioned above.

We now generalize the Hodge index theorem to Weil divisors on projective klt varieties using $\mathbb{Q}$-Chern classes.  
The following proposition is perhaps already known to experts, 
but we include an explanation for completeness, since appropriate references seem to be lacking.

\begin{prop}[Hodge index theorem] \label{prop-hodge-index-type}

Let $X$ be a projective klt variety of dimension $n$.  
Let $A$ and $B$ be $\Q$-Weil divisors on $X$, and let $H_1, \ldots, H_{n-2}$ be ample $\Q$-Cartier divisors on $X$.  

\begin{itemize}
\item[$(1)$] If $A^2 \cdot H_1 \cdots H_{n-2} > 0$, then we have 
\begin{equation}
\label{eqs-Hodge-Index-Theorem}
(A^2 \cdot H_1 \cdots H_{n-2}) \cdot (B^2 \cdot H_1 \cdots H_{n-2}) 
\leq (A \cdot B \cdot H_1 \cdots H_{n-2})^2.
\end{equation}

\item[$(2)$] If
$$
B^2 \cdot H_1 \cdots H_{n-2} = A \cdot B \cdot H_1 \cdots H_{n-2} = 0,
$$
then we have 
\begin{equation}
\label{eqs-Hodge-Index-Theorem-nu=1}
A^2 \cdot H_1 \cdots H_{n-2} \leq 0.
\end{equation}

\item[$(3)$] In cases $(1)$ and $(2)$, if equality holds in \eqref{eqs-Hodge-Index-Theorem} or \eqref{eqs-Hodge-Index-Theorem-nu=1}, respectively, then there exists a rational number $\lambda \in \Q$ such that 
$$
B \cdot L \cdot H_1 \cdots H_{n-2} = \lambda \cdot (A \cdot L \cdot H_1 \cdots H_{n-2})
$$
holds for any $\Q$-Weil divisor $L$ on $X$. 
\end{itemize}

\end{prop}

\begin{rem} \label{remin-Hodge-index-surfaces}
When $X$ is a smooth projective variety, this proposition follows directly from the Hodge index theorem (see, for example, \cite[Corollary~IV.2.15]{BPV84}).  
When both $A$ and $B$ are $\mathbb{Q}$-Cartier, the proposition can be proved by taking a resolution of singularities of $X$, even without assuming that $X$ has klt singularities.  
The main difference between this proposition and conventional formulations is its treatment of $\mathbb{Q}$-Weil divisors.
\end{rem}

\begin{proof}
We reduce the proof to the case where $X$ is a surface.  
The Weil divisors $A$ and $B$ are not necessarily $\Q$-Cartier, but this causes no difficulty thanks to \cite{Lan22a} and the assumption that $X$ has klt singularities.  

By replacing $H_i$ with $m_iH_i$ for $m_i \gg 1$, we may assume that each $H_i$ is a very ample Cartier divisor.  
Define the surface $S$ as the complete intersection 
$$
S := V_{1} \cap \cdots \cap V_{n-2},
$$
where each $V_i$ is a general member of $|H_i|$.  
Then, by \cite[Theorem~0.1]{Lan22a}, we obtain
$$
(A|_S)^2 = A^2 \cdot H_1 \cdots H_{n-2}. 
$$
Here, the intersection numbers are computed using $\Q$-Chern classes.  
The left-hand side $(A|_S)^2$ coincides with the usual intersection number of the $\Q$-Cartier divisor $A|_S$.  
Indeed, the surface $S$ has klt singularities and is therefore $\Q$-factorial, which implies that $A|_S$ is $\Q$-Cartier.  
The same argument applies to $A|_S \cdot B|_S$ and $(B|_S)^2$.  
Therefore, the desired conclusions follow from the standard Hodge index theorem (see Remark~\ref{remin-Hodge-index-surfaces}).
\end{proof}

\begin{lem}
\label{lem-Q-Cartier-of-Gi}
Let $X$ be a projective klt variety of dimension $n$. Let $\mathcal{L}$ be a reflexive sheaf of rank one  and $D$ be a  $\Q$-Cartier divisor on $X$ such that
$$\left(c_1(\mathcal{L}) - c_{1}( D)\right)   H_1\cdots  H_{n-1}= 0 \quad \text{and} \quad
\left(c_1(\mathcal{L}) - c_{1}(D) \right)^{2}   H_1\cdots  H_{n-2}= 0 \text{ holds  }
$$ 
for some ample $\Q$-Cartier divisors $H_{i}$ on $X$. 
Then, the sheaf $\mathcal{L}$ is a $\Q$-line bundle and satisfies $c_1(\mathcal{L}) = c_{1}(D)$.
\end{lem}

\begin{proof}
Let $\nu \colon \widehat{X} \rightarrow X$ be a maximally quasi-\'etale Galois cover with Galois group $G$, and let $m \in \mathbb{Z}_{>0}$ be a positive integer such that $mD$ is Cartier.  
The sheaf 
$$
\mathcal{M}':= \nu^{[*]}(\mathcal{L}^{[\otimes m]} \otimes \mathcal{O}_X(-mD))
$$ 
is reflexive and satisfies 
$$ 
c_1(\mathcal{M}') \cdot H_1 \cdots H_{n-1} = 0 
\quad \text{and} \quad 
c_1(\mathcal{M}')^{2} \cdot H_1 \cdots H_{n-2} = 0 
$$ 
by assumption.  
Furthermore, since $\mathcal{M}'$ has rank one, it is stable and satisfies $\widehat{c}_2(\mathcal{M}') \cdot H_1 \cdots H_{n-2} = 0$.  
Hence, by \cite[Theorem~1.4]{LT18}, we deduce that $\mathcal{M}'$ is a flat invertible sheaf.  

Replacing $m$ by a sufficiently divisible multiple, if necessary, we may assume that the natural action of $G$ on $\mathcal{M}'|_x$ is trivial for every $x \in \widehat{X}$. 
Consequently, there exists a flat line bundle $\mathcal{M}$ on $X$ such that $\mathcal{M}' = \nu^*\mathcal{M}$ (see \cite[Theorem~6.8]{Dre04}).  
This shows that $\mathcal{L}$ is a $\Q$-line bundle with $c_1(\mathcal{L}) = c_{1}(D)$.
\end{proof}

The following proposition is elementary, but keeping it in mind is often useful. 
The proof is straightforward, and thus we omit it.

\begin{prop} \label{prop-vanishing-intersection-form}
Let $X$ be a projective variety of dimension $n$, and let $0\leq k \leq n$ be an integer. Let $\alpha\colon {\rm N}^{1}(X)_{\Q}^{n-k} \to \Q$ be a symmetric $\Q$-multilinear form such that 
$$
\alpha  \cdot H_1 \cdots H_{k} \geq 0 
$$ 
holds for any ample $\Q$-divisors $H_i$ on $X$. Then, the following conditions are equivalent$:$

\begin{itemize}
\item[$\bullet$] $\alpha \cdot H_1 \cdots H_{k} = 0$ for some ample $\Q$-divisors $H_1, \ldots, H_{k}$. 
\item[$\bullet$] $\alpha \cdot H_1 \cdots H_{k} = 0$ for all ample $\Q$-divisors $H_1, \ldots, H_{k}$. 
\item[$\bullet$] $\alpha \cdot H_1 \cdots H_{k} = 0$ for all nef $\Q$-divisors $H_1, \ldots, H_{k}$.
\end{itemize}
\end{prop}

\subsection{Higgs sheaves}\label{subsec-higgs}

In this subsection, we establish some preliminary results on Higgs sheaves.  
We begin by recalling the definition and basic properties of Higgs sheaves.

\begin{defn}[{cf.\,\cite[Sections 4 and 5]{GKPT19b}}]
A {\textit{Higgs sheaf}} $(\mathcal{H}, \theta)$ on a normal projective variety $X$ 
is a pair consisting of a reflexive sheaf $\mathcal{H}$ and 
an $\mathcal{O}_{X}$-linear sheaf morphism $\theta \colon \mathcal{H} \to \mathcal{H} \otimes \Omega_{X}^{[1]}$,
called \emph{Higgs field},
such that the induced morphism 
$$\theta \wedge \theta \colon \mathcal{H} \to \mathcal{H} \otimes \Omega_{X}^{[2]}$$ vanishes. 
\begin{itemize}
\item[(1)]  A subsheaf $\mathcal{S} \subset \mathcal{H}$ is said to be \textit{generically $\theta$-invariant} 
if $\theta(\mathcal{S}|_{X_{\reg}})$ is contained in the image of the natural map $ \mathcal{S} \otimes \Omega_{X}^{[1]}  \to \mathcal{H} \otimes \Omega_{X}^{[1]}$ on $X_{\reg}$. 

\item[(2)] Let $L_1, \ldots,  L_{n-1}$ be nef $\Q$-Cartier divisors on $X$. The \emph{slope} of $(\mathcal{H}, \theta)$ with respect to $(L_1 \cdots L_{n-1})$ is defined by 
$$\mu_{L_1 \cdots L_{n-1}}(\mathcal{H}) := \frac{c_1(\mathcal{H}) \cdot L_1 \cdots L_{n-1}}{\rk \mathcal{H}}.$$
\item[(3)] The Higgs sheaf  $(\mathcal{H}, \theta)$ is said to be \textit{$(L_1 \cdots L_{n-1})$-semistable} (resp.\,\textit{$(L_1 \cdots L_{n-1})$-stable})
if any  generically $\theta$-invariant reflexive subsheaf $0  \subsetneq \mathcal{S} \subsetneq \mathcal{H}$ 
satisfies 
$$
\text{$\mu_{L_1 \cdots L_{n-1}}(\mathcal{S}) \le \mu_{L_1 \cdots L_{n-1}}(\mathcal{H})$ 
(resp.\,$\mu_{L_1 \cdots L_{n-1}}(\mathcal{S}) < \mu_{L_1 \cdots L_{n-1}}(\mathcal{H})$).}
$$ 
\end{itemize}
\end{defn}

\begin{prop}
\label{prop-Higgs-1}
Let $X$ be a projective klt variety of dimension $n$ and $L_1, \ldots,  L_{n-1}$ be nef $\Q$-Cartier divisors. 
For a reflexive subsheaf  $\mathcal{E} \subset \Omega_{X}^{[1]}$, 
we define the Higgs sheaf $(\mathcal{H}, \theta)$  by 
$\mathcal{H} := \mathcal{E} \oplus \mathcal{O}_{X}$ and 
$$
\begin{array}{cccc}
\theta \colon  & \mathcal{H} = \mathcal{E} \oplus \mathcal{O}_{X} &\longrightarrow  
&\mathcal{H} \otimes \Omega_{X}^{[1]} = (\mathcal{E} \oplus \mathcal{O}_{X}) \otimes \Omega_{X}^{[1]} \\
	   & (a,b) &\longmapsto& (0,a).
\end{array}
$$
If $\mathcal{E}$ is a $(L_1 \cdots L_{n-1})$-semistable sheaf with $\mu_{L_1 \cdots L_{n-1}}(\mathcal{E})>0$, 
then the Higgs sheaf $(\mathcal{H}, \theta)$ is $(L_1 \cdots L_{n-1})$-stable.
\end{prop}
\begin{proof}
Set $\alpha := L_{1} \cdots L_{n-1} \in {\rm N}^{1}(X)_{\Q}^{n-1}$.  
Suppose that $(\mathcal{H}, \theta)$ is not $\alpha$-stable.  
Then, there exists a non-trivial generically $\theta$-invariant subsheaf $\mathcal{S} \subset \mathcal{H}$ 
with $\mu_{\alpha}(\mathcal{S}) \ge \mu_{\alpha}(\mathcal{H})$.  
Set $r := \rk \mathcal{E}$ and $l := \rk \mathcal{S}$.  
By definition, we have 
$$
\mu_{\alpha}(\mathcal{S}) \ge \mu_{\alpha}(\mathcal{H}) 
= \frac{r}{r+1}\, \mu_{\alpha}(\mathcal{E}) > 0.
$$ 
Consider the sheaf morphism 
\begin{equation*}
\gamma \colon \mathcal{S} \subset \mathcal{H} = \mathcal{E} \oplus \mathcal{O}_{X} 
   \xrightarrow{\ \ \pr_{2}\ \ } \mathcal{O}_X
\end{equation*}
induced by the second projection.  
Note that $\gamma \colon \mathcal{S} \to \mathcal{O}_X$ is not the zero map, since $\mathcal{S}$ is generically $\theta$-invariant.  

We claim that $\Ker \gamma$ is a destabilizing subsheaf of $\mathcal{E}$ with respect to $\alpha$.  
First, we confirm that $\Ker \gamma$ is a non-zero subsheaf of $\mathcal{H}$.  
Indeed, if $\gamma$ were injective, then $\mathcal{S} \cong \Image \gamma \subset \mathcal{O}_X$ would be an ideal sheaf of $\mathcal{O}_X$.  
This would imply $0 \geq \mu_{\alpha}(\Image \gamma) = \mu_{\alpha}(\mathcal{S})$, which is a contradiction.  

We now consider the exact sequence 
$$
0 \to \Ker \gamma \to \mathcal{S} \to \Image \gamma \to 0. 
$$
Since $\Ker \gamma \subset \mathcal{S}$ and $\Image \gamma \subset \mathcal{O}_X$, 
both sheaves are torsion-free.  
It follows that 
\begin{align}\label{eq-slope}
c_{1}(\mathcal{S}) \cdot \alpha 
= c_{1}(\Ker \gamma) \cdot \alpha + c_{1}(\Image \gamma) \cdot \alpha 
\leq c_{1}(\Ker \gamma) \cdot \alpha. 
\end{align} 
By noting that $\rk(\Ker \gamma) = l - 1$ and using \eqref{eq-slope}, we obtain 
$$
\mu_{\alpha}(\Ker \gamma) 
= \frac{1}{l - 1}\, c_{1}(\Ker \gamma) \cdot \alpha
\ge \frac{l}{l - 1}\, \mu_{\alpha}(\mathcal{S})
\ge \frac{rl}{(r + 1)(l - 1)}\, \mu_{\alpha}(\mathcal{E}).
$$
By noting $\frac{rl}{(r + 1)(l - 1)} > 1$ holds, 
we obtain  $\mu_{\alpha}(\Ker \gamma) > \mu_{\alpha}(\mathcal{E})$, 
contradicting the $\alpha$-semistability of $\mathcal{E}$.
\end{proof}

\subsection{Numerically projectively flat sheaves}\label{subsec-flat}

In this section, we investigate numerically projectively flat reflexive sheaves, which play a crucial role in studying the singularities of varieties. We first review the definition of numerically projectively flat sheaves and the Bogomolov-Gieseker inequality.

\begin{defn}
Let $X$ be a projective klt variety of dimension $n$, and let $\mathcal{E}$ be a reflexive sheaf of rank $r$ on $X$. The sheaf $\mathcal{E}$ is said to be \textit{numerically projectively flat} if there exist ample $\mathbb{Q}$-Cartier divisors $H_1, \ldots, H_{n-1}$ such that $\mathcal{E}$ is $(H_1 \cdots H_{n-1})$-semistable and satisfies the Bogomolov-Gieseker equality
$$
\widehat{\Delta}(\mathcal{E}) H_1 \cdots H_{n-2} = 0, 
$$
where
$$
\widehat{\Delta}(\mathcal{E}) := 2r\, \widehat{c}_2(\mathcal{E}) - (r -1)\, \widehat{c}_{1}(\mathcal{E})^{2}.
$$
\end{defn}

\begin{prop}[cf.\,\cite{Lan22b}]
\label{prop-Higgs-2}
Let $X$ be a projective klt variety of dimension $n$. 
Let $(\mathcal{H}, \theta)$ be a reflexive Higgs sheaf on $X$, and let $H_1, \ldots, H_{n-1}$ be ample $\Q$-Cartier divisors on $X$. If the Higgs sheaf $(\mathcal{H}, \theta)$ is $(H_1 \cdots H_{n-1})$-semistable, then the Bogomolov-Gieseker inequality  
\begin{equation}
\label{eq-inequality-BG-Higgs}
\widehat{\Delta}(\mathcal{H}) H_1\cdots H_{n-2} \ge 0
\end{equation}
holds. 
Moreover, if $X$ is maximally quasi-\'etale and equality holds in \eqref{eq-inequality-BG-Higgs}, then $\End(\mathcal{H})$ is locally free.
\end{prop}

\begin{proof}
The Bogomolov-Gieseker inequality was proved in \cite[Theorem~7.6]{Lan22b}.  
Therefore, it suffices to prove the latter conclusion.  
Set $\mathcal{G} := \End(\mathcal{H})$, which is a reflexive Higgs sheaf equipped with the Higgs field $\theta_{\mathcal{G}}$ naturally induced by $\theta$.  
By \cite[Lemma~3.18]{GKPT19b}, the Higgs sheaf $(\mathcal{G}, \theta_{\mathcal{G}})$ is $(H_1 \cdots H_{n-1})$-semistable and satisfies equality in \eqref{eq-inequality-BG-Higgs}.  
Thus, since $c_1(\mathcal{G}) = 0$, we obtain
\begin{equation*}
\widehat{c}_2(\mathcal{G}) \cdot H_1 \cdots H_{n-2} = 0.
\end{equation*}

By \cite[Theorem~7.12]{Lan23}, it follows that $\mathcal{G}|_{X_{\reg}}$ is locally free and that $(\mathcal{G}, \theta_{\mathcal{G}})$ is semistable with respect to any polarization.  
Here, we used that Higgs fields correspond to integrable $\lambda$-connections in the case $\lambda=0$.  
By \cite[Theorem~6.6]{GKPT19a}, the Higgs bundle $(\mathcal{G}, \theta_{\mathcal{G}})|_{X_{\reg}}$ is induced from a linear representation 
\[
\rho_{0}\colon \pi_{1}(X_{\reg}) \to \GL(\rk \mathcal{G}, \C)
\]
via the nonabelian Hodge correspondence.  
This representation is extended to 
\[
\rho\colon \pi_{1}(X) \to \GL(\rk \mathcal{G}, \C)
\]
by Theorem~\ref{thm-representation-quasi-etale}.  
Applying the nonabelian Hodge correspondence once more, we obtain a locally free Higgs sheaf on $X$ corresponding to $\rho$.  
By reflexivity, this locally free Higgs sheaf on $X$ coincides with $(\mathcal{G}, \theta_{\mathcal{G}})$.
\end{proof}

\begin{cor}\label{cor-locfree}
Let $X$ be a maximally quasi-\'etale projective klt variety of dimension $n$, and let $\mathcal{E} \subset \Omega_{X}^{[1]}$ be a reflexive subsheaf of rank $r$ of the cotangent sheaf $\Omega_{X}^{[1]}$. Let $D$ be a nef $\Q$-Cartier divisor and $H_{1}, \ldots, H_{n-2}$ be ample $\Q$-Cartier divisors on $X$. Assume that $\mathcal{E}$ is $(D \cdot H_{1} \cdots H_{n-2})$-semistable, that $\mu_{D \cdot H_{1} \cdots H_{n-2}}(\mathcal{E} )>0$, and that
$$
\left( 2(r +1)\widehat{c}_2(\mathcal{E}) - r\, \widehat{c}_1(\mathcal{E})^2 \right) H_{1}\cdots H_{n-2}=0.
$$
Then, the sheaf $\mathcal{E}$ is locally free on $X$. 
\end{cor}

\begin{proof}
Consider the Higgs sheaf $(\mathcal{H} := \mathcal{E} \oplus \mathcal{O}_{X}, \theta)$ defined as in Proposition~\ref{prop-Higgs-1}.  
By Proposition~\ref{prop-Higgs-1}, the Higgs sheaf $(\mathcal{H}, \theta)$ is $(D \cdot H_{1} \cdots H_{n-2})$-stable.  
Then, by the argument of \cite[Proposition~4.17]{GKPT19b}, we deduce that $(\mathcal{H}, \theta)$ is also $((D + \varepsilon H_1) \cdot H_1 \cdots H_{n-2})$-stable for some $0 < \varepsilon \ll 1$.  
By assumption, the sheaf $\mathcal{H}$ satisfies the Bogomolov-Gieseker equality.  
Therefore, Proposition~\ref{prop-Higgs-2} shows that $\End(\mathcal{H})$ is locally free on $X$.  
Since $\mathcal{E}$ is a direct summand of
$$
\End(\mathcal{H}) = \End(\mathcal{E}) \oplus \mathcal{E} \oplus \mathcal{E}^\vee \oplus \mathcal{O}_X,
$$
it follows that $\mathcal{E}$ is locally free.
\end{proof}

\begin{lem}\label{lem-Proj-flat-bundles-are-Q-vectorbundles}
Let $X$ be a projective klt variety, and let $\mathcal{E}$ be a numerically projectively flat sheaf on $X$ such that $\det \mathcal{E}$ is a $\Q$-line bundle.  
Then, for any point $x \in X$, there exist an analytic open neighborhood $x \in U \subset X$ and a finite quasi-\'etale Galois cover 
$\pi \colon \widehat{U} \rightarrow U$ such that $\pi^{[*]} \mathcal{E}$ is locally free.
\end{lem}

\begin{proof}
Fix a point $x \in X$.  
By \cite[Theorem~1]{Bra21}, we can take an analytic open neighborhood $x \in U \subset X$ such that the \emph{regional fundamental group} $\pi_1(U_{\reg})$ is finite.  
By \cite[Proposition~3.13]{GKP16}, there exists a finite quasi-\'etale cover $p \colon V \to U$ such that $\pi_1(V_{\reg}) = \{\mathrm{id}\}$.  
Note that $\pi_1(V) = \{\mathrm{id}\}$ by the normality of $V$.  
Then, by \cite[Proposition~3.11]{GKP22}, there exists a reflexive sheaf $\mathcal{L}$ of rank one such that 
\[
p^{[*]}\mathcal{E} \cong \mathcal{L}^{\oplus r}.
\]

Since $\mathcal{L}^{[\otimes r]} \cong p^{[*]}\det \mathcal{E}$ is a $\Q$-line bundle by assumption, the sheaf $\mathcal{L}^{[\otimes m]}$ is locally free for some $m \in \mathbb{N}$.  
After shrinking $U$ if necessary, we may assume that $\mathcal{L}^{[\otimes m]} \cong \mathcal{O}_V$ is trivial.  
Let 
$$\widehat{p} \colon \widehat{U} \to V
$$ be the associated index-one cover such that $\widehat{p}^{[*]}\mathcal{L} = \mathcal{O}_{\widehat{U}}$, which is a finite quasi-\'etale cyclic cover of order $m$.  
Then, since $\pi_1(V_{\reg}) = \{\mathrm{id}\}$, this cover is actually the identity map.  
Finally, the Galois closure of $p \circ \widehat{p} \colon \widehat{U} \to U$ satisfies the desired conclusion.
\end{proof}

\subsection{Filtrations of reflexive sheaves}\label{subsec-filtration}

In this section, we study filtrations of reflexive sheaves in detail.  
Although the discussions and results are technical, they provide essential insights for this paper.  
The following lemma investigates the gap between torsion-free sheaves and their reflexive hulls.

\begin{lem}[{cf.\,\cite[Lemma 2.3]{Kaw92}}]
\label{lem-c2-in-short-exact-seq}
Let $X$ be a projective klt variety of dimension $n$.  
Consider an exact sequence of sheaves 
\begin{align}\label{eq-exact-a}
0 \longrightarrow \mathcal{F} \longrightarrow \mathcal{E} \longrightarrow \mathcal{Q} \longrightarrow 0,
\end{align}
where $\mathcal{F}$ and $\mathcal{E}$ are reflexive and $\mathcal{Q}$ is torsion-free.  
Then, for any ample $\mathbb{Q}$-Cartier divisors $H_{i}$, the following inequality holds$:$ 
\begin{equation}\label{eq-c2-in-short-exact-seq}
\bigl( \widehat{c}_2(\mathcal{E}) - \widehat{c}_2(\mathcal{F}) - \widehat{c}_2(\mathcal{Q}^{\vee \vee}) 
- c_1(\mathcal{F}) c_1(\mathcal{Q}) \bigr) \cdot H_1 \cdots H_{n-2} \geq 0.
\end{equation}

Moreover, if equality holds in \eqref{eq-c2-in-short-exact-seq} for some ample $\mathbb{Q}$-Cartier divisors $H_{i}$, 
then there exists a Zariski open subset $X^\circ \subset X$ with $\codim(X \setminus X^\circ) \geq 3$ such that the sequence \eqref{eq-exact-a} is Zariski-locally split on $X^\circ$.  
In particular, in this case, the natural morphism $\mathcal{Q} \to \mathcal{Q}^{\vee \vee}$ is an isomorphism over $X^\circ$.
\end{lem}

\begin{proof}
We first reduce the proof to the case where $X$ is a surface.  
Let $S$ be the surface defined by the complete intersection
\[
S := V_{1} \cap \cdots \cap V_{n-2},
\]
where $V_{i}$ is a general member of $|m_i H_i|$ for $m_i \gg 1$.  
Then the sheaves $\mathcal{F}|_{S}$, $\mathcal{E}|_{S}$, and $\mathcal{Q}|_{S}$ satisfy the assumptions of the lemma.  
Indeed, a general member $V_{i}$ does not contain the associated primes of the relevant sheaves, so the restriction of \eqref{eq-exact-a} to $S$ remains exact, and reflexivity of sheaves on $X$ is preserved under restriction to $S$.  
Furthermore, by property (P2), the left-hand side of \eqref{eq-c2-in-short-exact-seq} coincides with
\[
\widehat{c}_2(\mathcal{E}|_{S}) - \widehat{c}_2(\mathcal{F}|_{S}) 
- \widehat{c}_2(\mathcal{Q}^{\vee \vee}|_{S}) 
- c_1(\mathcal{F}|_{S}) c_1(\mathcal{Q}|_{S}).
\]
Finally, if the sequence splits locally after restriction to $S$, then it also splits locally in a neighborhood of $S$.  
Therefore, we may assume that $X$ is a surface.

As in property (P1), we can take a finite Galois cover $\pi \colon \widehat{X} \to X$ such that $\pi^{[*]}\mathcal{E}$ is locally free on $\widehat{X}$ and satisfies
\begin{equation}
\label{eq-c2-in-short-exact-seq-1}
\bigl( \pi_* \pi^{[*]}\mathcal{E} \bigr)^G = \mathcal{E} 
\quad \text{and} \quad
c_2(\pi^{[*]}\mathcal{E}) = d \cdot \widehat{c}_2(\mathcal{E}),
\end{equation}
where $G := \mathrm{Gal}(\widehat{X}/X)$ and $d := \deg \pi = |G|$.  
We may assume that the same properties also hold for $\mathcal{F}$ and $\mathcal{Q}^{\vee \vee}$.  

Consider the induced exact sequence 
\begin{equation}\label{exact-seq-tor}
0 \longrightarrow \pi^{[*]}\mathcal{F} 
\longrightarrow \pi^{[*]}\mathcal{E} 
\overset{\phi}{\longrightarrow} \pi^{[*]}(\mathcal{Q}^{\vee \vee}) 
\longrightarrow \mathcal{T} := \Coker(\phi) 
\longrightarrow 0,
\end{equation}
where $\mathcal{T}$ is a skyscraper sheaf on $\widehat{X}$.  
By \cite[Proposition~2.1]{Mum83} or \cite[Lemma~10.5]{K++}, the sheaf $\mathcal{T}$ admits a finite locally free resolution on $\widehat{X}$, which allows us to define $c_2(\mathcal{T})$.  
Moreover, by \cite[Lemma~3.3]{Mum83} or \cite[Lemma~10.9]{K++}, we have that $-c_2(\mathcal{T})$ equals the length of $\mathcal{T}$ (hence non-negative).  
Thus, using \eqref{eq-c2-in-short-exact-seq-1}, we deduce
\begin{align*}
d \cdot \widehat{c}_2(\mathcal{E}) 
&= c_2(\pi^{[*]}\mathcal{E}) \\
&= c_2(\pi^{[*]}\mathcal{F}) 
+ c_2(\pi^{[*]}\mathcal{Q}^{\vee \vee}) 
+ c_1(\pi^{[*]}\mathcal{F}) c_1(\pi^{[*]}\mathcal{Q}^{\vee \vee})
- c_2(\mathcal{T}) \\
&\geq c_2(\pi^{[*]}\mathcal{F}) 
+ c_2(\pi^{[*]}\mathcal{Q}^{\vee \vee}) 
+ c_1(\pi^{[*]}\mathcal{F}) c_1(\pi^{[*]}\mathcal{Q}^{\vee \vee}) \\
&= d \cdot \Bigl( \widehat{c}_2(\mathcal{F}) 
+ \widehat{c}_2(\mathcal{Q}^{\vee \vee}) 
+ \widehat{c}_1(\mathcal{F}) \widehat{c}_1(\mathcal{Q}^{\vee \vee}) \Bigr).
\end{align*}
This proves \eqref{eq-c2-in-short-exact-seq}.

Assume that equality holds in \eqref{eq-c2-in-short-exact-seq}. 
Then, the skyscraper sheaf $\mathcal{T}$ vanishes, which means that 
\begin{equation}
\label{eq-exact-local-free}
0 \longrightarrow \pi^{[*]}\mathcal{F} 
\longrightarrow \pi^{[*]}\mathcal{E} 
\overset{\phi}{\longrightarrow} \pi^{[*]}(\mathcal{Q}^{\vee \vee}) 
\longrightarrow 0
\end{equation}
is exact.  
Since all sheaves in \eqref{eq-exact-local-free} are locally free, the sequence is Zariski locally split.  
By pushing forward \eqref{eq-exact-local-free} via $\pi$ and taking $G$-invariants, we deduce that 
\begin{equation}
\label{eq-exact-og'}
0 
\longrightarrow \bigl(\pi_* \pi^{[*]}\mathcal{F}\bigr)^G 
\longrightarrow \bigl(\pi_* \pi^{[*]}\mathcal{E}\bigr)^G
\longrightarrow \bigl(\pi_* \pi^{[*]}(\mathcal{Q}^{\vee \vee})\bigr)^G
\longrightarrow 0
\end{equation}
is exact.  
Note that $\bigl(\pi_*(\bullet)\bigr)^G$ is exact since $\pi$ is a finite morphism and $G$ is finite (cf.\,\cite[Lemma~B.3]{GKKP11}).  
Moreover, since $\pi$ is finite, $\pi_*$ preserves reflexivity of sheaves.  
Hence the natural morphism 
\[
\mathcal{F} \longrightarrow \bigl(\pi_* \pi^{[*]}\mathcal{F}\bigr)^G
\]
is an isomorphism, since it is clearly an isomorphism on the locally free locus of $\mathcal{F}$ (which has codimension $\geq 2$).  
The same statement holds for $\mathcal{E}$ and $\mathcal{Q}^{\vee \vee}$.  
Thus we obtain the exact sequence
\begin{equation}
\label{eq-exact-og}
0 \longrightarrow \mathcal{F} \longrightarrow \mathcal{E} 
\longrightarrow \mathcal{Q}^{\vee \vee} \longrightarrow 0.
\end{equation}
This shows that $\mathcal{Q} \to \mathcal{Q}^{\vee \vee}$ is an isomorphism.  
Moreover, if $\sigma \colon \pi^{[*]} \mathcal{Q}|_U \to \pi^{[*]} \mathcal{E}|_U$ gives a splitting of \eqref{eq-exact-local-free} over a $G$-invariant open subset $U \subset \widehat{X}$, then
\[
\frac{1}{|G|} \sum_{g \in G} g^* \sigma \colon 
\mathcal{Q}|_U \cong \bigl(\pi_* \pi^{[*]} \mathcal{Q}|_U \bigr)^G 
\longrightarrow \bigl(\pi_* \pi^{[*]} \mathcal{E}|_U \bigr)^G 
\cong \mathcal{E}|_U
\]
yields a splitting of \eqref{eq-exact-og} over $\pi(U) \subset X$.
\end{proof}

The following lemma, although technical, plays a crucial role in what follows.

\begin{lem}\label{lem-(co)tangent-extension-of-proj-flat}
Let $X$ be a projective klt variety of dimension $n$ and let $H_1, \ldots, H_n$ be ample divisors on $X$. Consider an exact sequence of sheaves 
\begin{align}\label{eq-exact-b}
0 \longrightarrow \mathcal{F} 
\longrightarrow \mathcal{E} 
\longrightarrow \mathcal{Q} 
\longrightarrow 0, 
\end{align}
where $\mathcal{F}$ and $\mathcal{E}$ are reflexive and $\mathcal{Q}$ is torsion-free. Assume that $\mathcal{F}$ and $\mathcal{Q}^{\vee \vee}$ are numerically projectively flat and that
$\det(\mathcal{F})$ and $\det(\mathcal{Q})$ are $\Q$-line bundles. If
$$
\left( \widehat{c}_2(\mathcal{E}) - \widehat{c}_2(\mathcal{F}) -  \widehat{c}_2(\mathcal{Q}^{\vee \vee}) -  c_1(\mathcal{F}) c_1(\mathcal{Q}) \right)\cdot H_1 \cdots  H_{n-2} 
 = 0,
$$
holds, 
then the sequence \eqref{eq-exact-b} is analytically locally split on $X$. 
In particular, the following statements hold$:$
\begin{itemize}
\item[$(1)$] $\mathcal{Q}$ is reflexive, and 
\item[$(2)$] the dual sequence 
\begin{align}\label{eq-(co)tangent-extension-of-proj-flat-0}
0 \rightarrow \mathcal{Q}^\vee \rightarrow \mathcal{E}^\vee \rightarrow \mathcal{F}^\vee \rightarrow 0
\end{align}
is exact.
\end{itemize}
\end{lem}

\begin{proof}
Fix a point $x \in X$.  
By Lemma~\ref{lem-Proj-flat-bundles-are-Q-vectorbundles}, 
there exist an analytic open set $U \subset X$ and a finite quasi-\'etale Galois cover 
$\pi \colon V \to U$ 
such that both $\pi^{[*]}\mathcal{F}$ and $\pi^{[*]}\mathcal{Q}$ are locally free on $V$.  

We first show that the induced sequence 
\begin{equation}\label{eq-exact}
0 \longrightarrow \pi^{[*]}\mathcal{F} 
\longrightarrow \pi^{[*]}\mathcal{E} 
\longrightarrow \pi^{[*]}\mathcal{Q} 
\longrightarrow 0
\end{equation}
is an exact sequence of vector bundles on $V$.  
The exactness on the left is immediate.  
By Lemma~\ref{lem-c2-in-short-exact-seq}, the sequence \eqref{eq-exact} is exact (indeed, even locally split) in codimension two.  
Note that $V$ has klt singularities since it is a finite quasi-\'etale cover of the klt variety $U$; hence $V$ is Cohen-Macaulay.  
Applying \cite[Lemma~9.9]{AD14}, we conclude that \eqref{eq-exact} is an exact sequence of vector bundles.  
Strictly speaking, \cite[Lemma~9.9]{AD14} is stated only for schemes, but its proof applies equally to complex analytic varieties.  

Let $G$ be the Galois group of $\pi \colon V \to U$.  
Then, as in the second half of the proof of Lemma~\ref{lem-c2-in-short-exact-seq}, the sequence
\begin{equation}\label{eq-isom2}
0 
\longrightarrow \bigl(\pi_* \pi^{[*]}\mathcal{F}\bigr)^G 
\longrightarrow \bigl(\pi_* \pi^{[*]}\mathcal{E}\bigr)^G
\longrightarrow \bigl(\pi_* \pi^{[*]}\mathcal{Q}\bigr)^G
\longrightarrow 0
\end{equation}
is exact, analytically locally split on $U$, and can be identified with the exact sequence
\begin{equation*}
0 \longrightarrow \mathcal{F} 
\longrightarrow \mathcal{E} 
\longrightarrow \mathcal{Q}^{\vee \vee} 
\longrightarrow 0.
\end{equation*}
This shows that $\mathcal{Q}$ is reflexive and that \eqref{eq-exact-b} is analytically locally split.  
\end{proof}

\section{Harder-Narasimhan filtrations of generically nef reflexive sheaves}
\label{section:3}
In this subsection, we extend to klt varieties a well-known relation between $\widehat{\Delta}(\mathcal{E})$ and the Harder-Narasimhan filtration of $\mathcal{E}$, originally established for smooth varieties (see, for example, \cite{Lan04}).  
Throughout this section, we work under the following setup: 

\begin{setup}
\label{setup-torsionfree-Langer}

Let $X$ be a projective klt variety of dimension $n$, and let $\mathcal{E}$ be a reflexive sheaf of rank $r$ on $X$.  
Let $D$ be a nef $\mathbb{Q}$-Cartier divisor, and let $H_1, \ldots, H_{n-2}$ be ample $\mathbb{Q}$-Cartier divisors on $X$.  
Following \cite[Corollary~2.27]{GKP16a}, we consider the $(D \cdot H_1 \cdots H_{n-2})$-Harder-Narasimhan filtration of $\mathcal{E}$:
\[
0 =: \mathcal{E}_0 \subsetneq \mathcal{E}_1 \subsetneq \cdots \subsetneq \mathcal{E}_l := \mathcal{E}.
\]
We adopt the following notation:
\begin{itemize}
    \item $\mathcal{G}_i := \mathcal{E}_i / \mathcal{E}_{i-1}$
    \quad  and \quad $r_i := \operatorname{rk}(\mathcal{G}_i)$,
    \item $\alpha := D \cdot H_1 \cdots H_{n-2}$,
    \item $\mu_i := \mu_{\alpha}(\mathcal{G}_i)$ \quad  and \quad $\mu := \mu_{\alpha}(\mathcal{E})$,
    \item $\widehat{\Delta}(\mathcal{E}_i) := 2 r_i \widehat{c}_2(\mathcal{E}_i) - (r_i - 1)\widehat{c}_1(\mathcal{E}_i)^2$.
\end{itemize}
By construction, the graded pieces $\mathcal{G}_i$ are torsion-free semistable sheaves with $\mu_i > \mu_{i+1}$.  
The \emph{maximum slope} $\mu_1 = \mu_{\alpha}^{\max}(\mathcal{E})$ is the supremum of $\mu_{\alpha}(\mathcal{F})$, where $\mathcal{F}$ ranges over all non-zero subsheaves $\mathcal{F} \subset \mathcal{E}$.  
Similarly, the \emph{minimal slope} $\mu_l = \mu_{\alpha}^{\min}(\mathcal{E})$ is the infimum of $\mu_{\alpha}(\mathcal{Q})$, where $\mathcal{Q}$ ranges over all torsion-free quotients $\mathcal{E} \twoheadrightarrow \mathcal{Q}$.  
The first graded piece $\mathcal{E}_1$ is called the \emph{maximal destabilizing subsheaf} of $\mathcal{E}$.
\end{setup}

We begin with the following elementary lemmas.  
Although these results were previously proved in \cite{Lan04},  
we include the proofs here for completeness, with particular attention to the conditions under which equality holds.

\begin{lem}[{cf.\,\cite[Lemma 1.4]{Lan04}}]\label{lem-Langer-inequality}
Let $l \in \mathbb{Z}_{+}$, and let $r, \mu, r_1, \dots, r_l, \mu_1, \ldots, \mu_l \in \mathbb{R}_{+}$
satisfy the following conditions$:$ 
\begin{enumerate}
    \item[$(1)$] $r = r_1 + \cdots + r_l$.
    \item[$(2)$] $\mu_1 > \cdots > \mu_l \ge 0$.
    \item[$(3)$] $r\mu = r_1 \mu_1 + \cdots + r_l \mu_l$.
\end{enumerate}
Then, the following inequality holds$:$
\begin{equation}
\label{eq-Langer-inequality1}
\sum_{1 \le i < j \le l} r_i r_j (\mu_i - \mu_j)^2 
\le r^2(\mu_1 - \mu)(\mu - \mu_l).
\end{equation}
Moreover, if equality  holds in \eqref{eq-Langer-inequality1}, then $l \le 2$. 
\end{lem}

\begin{proof}
By rescaling $r_i$ to $r_i/r$ and $\mu_i$ to $\mu_i/\mu$, we may assume $r = \mu = 1$.  
From conditions~(1) and~(3), we obtain 
\begin{align*}
\sum_{1 \leq i < j \leq l} r_i r_j (\mu_i - \mu_j)^2 
&= \frac{1}{2} \sum_{i=1}^{l} \sum_{j=1}^{l} r_i r_j (\mu_i - \mu_j)^2 \\
&= \sum_{i=1}^{l} r_i \mu_{i}^2 - 1.
\end{align*}
Using condition~(2), we estimate the right-hand side as follows:
\begin{align*}
\sum_{i=1}^{l} r_i \mu_{i}^2 - 1 
&= \sum_{i=1}^{l} r_i (\mu_{i} + \mu_{l})(\mu_{i} - \mu_{l}) - 1 + \mu_{l}^{2} \\
&\leq (\mu_{1} + \mu_{l}) \sum_{i=1}^{l} r_i (\mu_{i} - \mu_{l}) - 1 + \mu_{l}^{2} \\
&= (\mu_{1} + \mu_{l})(1 - \mu_{l}) - 1 + \mu_{l}^{2} 
\quad \text{(since $\sum_{i=1}^{l} r_i = 1$ and $\sum_{i=1}^{l} r_i \mu_i = 1$)} \\
&= (\mu_{1} - 1)(1 - \mu_{l}).
\end{align*}
This proves the desired inequality \eqref{eq-Langer-inequality1}.  

If equality  holds in \eqref{eq-Langer-inequality1}, the above argument shows that 
\[
r_i (\mu_i - \mu_1)(\mu_i - \mu_l) = 0 \quad \text{for all } i = 1, \ldots, l.
\]
By condition~(2), this implies $l \leq 2$.
\end{proof}

\begin{lem}[{cf.\,\cite[Theorem 5.1]{Lan04}}]
\label{lem-Langer-Theorem:5.1}
Consider Setup \ref{setup-torsionfree-Langer}. Then, the following inequality holds$:$
\begin{align}
\begin{split}
\label{eq-Langer-Theorem:5.1}
\frac{\widehat{\Delta}(\mathcal{E}) }{r} H_1 \cdots H_{n-2} \
 &\ge \ \sum_{i=1}^{l} \frac{\widehat{\Delta}(\mathcal{G}_{i}^{\vee\vee})}{r_i} H_1 \cdots H_{n-2} \\
&\quad - \frac{1}{r} \sum_{1 \le i < j \le l} r_i r_j \left( \frac{c_1(\mathcal{G}_{i})}{r_i} - \frac{c_1(\mathcal{G}_{j})}{r_j} \right)^2 H_1 \cdots H_{n-2}.
\end{split}
\end{align}
Moreover, if equality  holds in \eqref{eq-Langer-Theorem:5.1}, 
then the natural morphism $\mathcal{G}_{i} \to \mathcal{G}_{i}^{\vee\vee}$ is isomorphism in codimension two for any $i=1, \ldots, l$.
\end{lem}

\begin{proof}
In the case $l=2$, this lemma follows directly from Lemma~\ref{lem-c2-in-short-exact-seq}.  
In the general case, by applying Lemma~\ref{lem-c2-in-short-exact-seq} inductively,  
we obtain the desired conclusion (see \cite[Chapter~5]{Lan04}).
\end{proof}

With the above preparations in place, we are now ready to prove the following proposition:

\begin{prop}[{cf.\,\cite[Theorem 5.1]{Lan04}}]
\label{prop-Langer-inequality}
Consider Setup \ref{setup-torsionfree-Langer} and assume that $d := D^{2} H_1 \cdots H_{n-2} > 0$. 
Then, the following inequality holds$:$
\begin{equation}
\label{eq-Langer-inequality}
\frac{\widehat{\Delta}(\mathcal{E}) }{r} H_1 \cdots H_{n-2}
 \ge - \frac{r}{d} (\mu_1 - \mu)(\mu - \mu_{l}).
\end{equation}
Moreover, if equality holds in \eqref{eq-Langer-inequality}, 
then we have$:$
\begin{enumerate}[label=$(\arabic*)$]
\item $l \le 2$, and $\widehat{\Delta}(\mathcal{G}_{i}^{\vee\vee}) H_1 \cdots H_{n-2} = 0$ for any $i =1,2$.
\item There exists $\lambda \in \Q$ such that 
$$
\frac{c_1(\mathcal{G}_1)}{r_1} - \frac{c_1(\mathcal{G}_2)}{r_2}
\equiv \lambda D.
$$
\item The natural morphism $\mathcal{G}_{2} \to \mathcal{G}_{2}^{\vee\vee}$ is an isomorphism in codimension two.
\end{enumerate}
\end{prop}

\begin{proof}
To simplify notation, we set $\beta := H_1 \cdots H_{n-2}$. 
The semistability of $\mathcal{G}_{i}^{\vee\vee}$ implies that 
the Bogomolov-Gieseker inequality holds:
\begin{equation}
\label{eq-Bogomolov-Gieseker-klt}
\widehat{\Delta}(\mathcal{G}_{i}^{\vee\vee}) \beta \ge 0.
\end{equation}
The Hodge index theorem (see Proposition \ref{prop-hodge-index-type}) yields
\begin{equation}
\label{eq-Langer-Hodge}
\frac{1}{r} \sum_{1 \le i < j \le l} r_i r_j \left( \frac{c_1(\mathcal{G}_i)}{r_i} - \frac{c_1(\mathcal{G}_j)}{r_j} \right)^2 \beta \le \frac{1}{r d} \sum_{1 \le i < j \le l} r_i r_j (\mu_i - \mu_j)^2.
\end{equation}
Applying Lemmas \ref{lem-Langer-inequality} and \ref{lem-Langer-Theorem:5.1}, together with inequalities \eqref{eq-Langer-Theorem:5.1} and \eqref{eq-Langer-Hodge}, we obtain the first conclusion:
\begin{align}
\begin{split}
\label{eq-Langer-Delta1.5}
\frac{\widehat{\Delta}(\mathcal{E}) }{r} \beta &\underalign{\text{(by Lem. \ref{lem-Langer-Theorem:5.1})}}{\ge}  \quad  \sum_{i=1}^{l} \frac{\widehat{\Delta}(\mathcal{G}_{i}^{\vee\vee})}{r_i} \beta - \frac{1}{r} \sum_{1 \le i < j \le l} r_i r_j \left( \frac{c_1(\mathcal{G}_i)}{r_i} - \frac{c_1(\mathcal{G}_j)}{r_j} \right)^2 \beta \quad  \\
&\underalign{\text{(by \eqref{eq-Bogomolov-Gieseker-klt})}}{\ge} \quad -\frac{1}{r} \sum_{1 \le i < j \le l} r_i r_j \left( \frac{c_1(\mathcal{G}_i)}{r_i} - \frac{c_1(\mathcal{G}_j)}{r_j} \right)^2 \beta  \\
&\underalign{\text{(by \eqref{eq-Langer-Hodge})}}{\ge} \quad - \frac{1}{r d} \sum_{1 \le i < j \le l} r_i r_j (\mu_i - \mu_j)^2  \\
&\underalign{\text{(by Lem. \ref{lem-Langer-inequality})}}{\ge}  \quad - \frac{r}{d} (\mu_1 - \mu)(\mu - \mu_l).
\end{split}
\end{align}

For the latter conclusion, we assume that equality holds in \eqref{eq-Langer-inequality}.
Conclusions (1) and (3) directly follow from Lemmas \ref{lem-Langer-Theorem:5.1}, \ref{lem-Langer-inequality}, and \eqref{eq-Bogomolov-Gieseker-klt}. 
Moreover, the Hodge index theorem shows that 
$$
\left( \frac{c_1(\mathcal{G}_1)}{r_1} - \frac{c_1(\mathcal{G}_2)}{r_2} \right) \cdot L \cdot \beta = \lambda D \cdot L \cdot \beta
$$
for any Weil divisor $L$ on $X$. Hence, conclusion~(2) follows from  Lemma~\ref{lem-Q-Cartier-of-Gi}. 
\end{proof}

\subsection{Case of \texorpdfstring{$\nu(\det \mathcal{E}) \geq 2$}{}}\label{subsec-case1}

In this subsection, we describe the structure of the Harder-Narasimhan filtration 
for a generically nef reflexive sheaf $\mathcal{E}$ with nef determinant $\det \mathcal{E}$ 
in the case where $\nu(\det \mathcal{E}) \geq 2$.  
Here, a reflexive sheaf $\mathcal{E}$ is called \emph{generically nef} 
if $\mu_{H_1 \cdots H_{n-1}}^{\min}(\mathcal{E}) \geq 0$ for all $\mathbb{Q}$-ample Cartier divisors $H_1, \ldots, H_{n-1}$.

\begin{thm}[{cf.\,\cite[Corollary 3.5]{LL23}}]
\label{thm-Langer-Liu-Liu-inequality}
Let $H_{1}, \ldots, H_{n-2}$ be ample $\Q$-Cartier divisors on a projective klt variety $X$, and let $\mathcal{E}$ be a reflexive sheaf on $X$ satisfying the following conditions$:$
\begin{itemize}
\item $\mathcal{E}$ is generically nef. 
\item $\det \mathcal{E}$ is a nef $\Q$-line bundle. 
\item $\nu(\det \mathcal{E}) \geq 2$. 
\end{itemize}
Let $r_1$ be the rank of the maximal destabilizing subsheaf $\mathcal{G}_{1}:=\mathcal{E}_{\max}$ of $\mathcal{E}$ 
with respect to $c_1(\mathcal{E}) \cdot H_1 \cdots H_{n-2}$. Then, the following inequality holds$:$
\begin{equation}
\label{eq-Langer-Liu-Liu-inequality}
\widehat{c}_2(\mathcal{E}) H_1 \cdots H_{n-2} 
\ge \frac{r_1 -1}{2 r_1} c_1(\mathcal{E})^{2} H_1 \cdots H_{n-2}. 
\end{equation}

Moreover, if equality holds in \eqref{eq-Langer-Liu-Liu-inequality}, 
then the $(c_1(\mathcal{E}) \cdot H_1 \cdots H_{n-2})$-Harder-Narasimhan filtration of $\mathcal{E}$ is given by 
\[
0 \longrightarrow \mathcal{G}_1 = \mathcal{E}_{\max} 
\longrightarrow \mathcal{E} 
\longrightarrow \mathcal{G}_2 
\longrightarrow 0
\]
and satisfies the following properties$:$
\begin{enumerate}[label=$(\arabic*)$]
\item $\det \mathcal{G}_1$ is a $\Q$-line bundle such that $c_1(\mathcal{G}_1) = c_1(\mathcal{E})$ and $\widehat{\Delta}(\mathcal{G}_1) H_1 \cdots H_{n-2} = 0$. 
\item $c_1(\mathcal{G}_2) = 0$ and $\widehat{\Delta}(\mathcal{G}_{2}^{\vee\vee}) H_1 \cdots H_{n-2} = 0$ holds.
\item The natural morphism $\mathcal{G}_{2} \to \mathcal{G}_{2}^{\vee\vee}$ is an isomorphism in codimension two.
\end{enumerate}
In particular, if $X$ is maximally quasi-\'etale, then $\mathcal{G}_{2}^{\vee\vee}$ is a flat locally free sheaf.
\end{thm}
\begin{proof}
We consider the same situation as in Setup \ref{setup-torsionfree-Langer} in the case $D := c_1(\mathcal{E})$. 
Since $\mathcal{E}$ is generically nef, we have 
\begin{equation}
\label{eq-generically-nef}
r_1 \mu_1 \ge r \mu \quad \text{and} \quad \mu_{l} \ge 0.  
\end{equation}
From now on, we set 
$$d := c_1(\mathcal{E})^2 H_1 \cdots H_{n-2} \quad \text{and} \quad \beta := H_1 \cdots H_{n-2}. $$ 
Note that $0<d = c_1(\mathcal{E})^2 \beta = r \mu $ holds. 
Then, the desired inequality \eqref{eq-Langer-Liu-Liu-inequality} follows from 
\begin{align}
\begin{split}
\label{eq-Langer-Delta2.5-corrected}
\frac{\widehat{\Delta}(\mathcal{E}) }{r} \beta
& \underalign{\text{(by Prop.\ \ref{prop-Langer-inequality}})}{\ge}  
- \frac{r}{d} (\mu_1 - \mu)(\mu - \mu_l) \\
&\underalign{\text{(by \eqref{eq-generically-nef})}}{\ge}  
- \frac{r}{d} (\mu_1 - \mu) \mu \\
&\underalign{\text{(by \eqref{eq-generically-nef})}}{\ge}  
- \frac{r}{d} \left( \frac{r_1 \mu_1 - r \mu}{r_1} \right) \mu \\
&\underalign{}{=} 
- \frac{(r_1 \mu_1 - r \mu)}{r_1} \quad (\text{since } d = r \mu) \\
&\underalign{}{=} 
\left( \frac{1}{r} - \frac{1}{r_1} \right) c_{1}(\mathcal{E})^{2} \beta.
\end{split}
\end{align}
Assume that equality holds in \eqref{eq-Langer-Liu-Liu-inequality}.
Proposition \ref{prop-Langer-inequality} (1) shows that $l \le 2$. 
In the case $l=1$, by setting $\mathcal{G}_1 = \mathcal{E}$, we complete the proof.
Thus, we may assume that $l=2$. 

We now show that $\det \mathcal{G}_2$ is a $\Q$-line bundle with $c_1(\mathcal{G}_2) = 0$.
From \eqref{eq-generically-nef}, we can deduce that $\mu_{2} = 0$, 
and thus $c_1(\mathcal{G}_2) c_1(\mathcal{E}) \beta = 0$.
This implies that 
\begin{equation}
\label{eq-hodge-1}
\left( c_1(\mathcal{G}_1) + c_1(\mathcal{G}_2) \right) c_1(\mathcal{G}_2) \beta = 0.
\end{equation}
On the other hand, by Proposition \ref{prop-Langer-inequality} (2), we obtain
\begin{equation}
\label{eq-hodge-2}
\left( \frac{c_1(\mathcal{G}_1)}{r_1} - \frac{c_1(\mathcal{G}_2)}{r_2} \right) c_1(\mathcal{G}_2) \beta = 0.
\end{equation}
Subtracting \eqref{eq-hodge-2} from \eqref{eq-hodge-1}, we find that
$$
c_1(\mathcal{G}_1) c_1(\mathcal{G}_2) \beta = c_1(\mathcal{G}_{2})^2 \beta = 0.
$$
By applying the Hodge index theorem (see Proposition \ref{prop-hodge-index-type}) to $A = c_1(\mathcal{E})$ and $B = c_1(\mathcal{G}_2)$,
we deduce that $\det \mathcal{G}_2$ is a $\Q$-line bundle with $c_1(\mathcal{G}_2) = 0$ (see Lemma \ref{lem-Q-Cartier-of-Gi}). 
This shows that $\det \mathcal{G}_1$ is a $\Q$-line bundle and that $c_1(\mathcal{G}_1) = c_1(\mathcal{E})$. 
Moreover, Proposition~\ref{prop-Langer-inequality}\,(1) yields 
$\widehat{\Delta}(\mathcal{G}_{1}^{\vee\vee}) \cdot \beta = \widehat{\Delta}(\mathcal{G}_{2}^{\vee\vee}) \cdot \beta = 0$. 
This completes the proof of (1)--(3).

Assume further that $X$ is maximally quasi-\'etale. 
We now show that $\mathcal{G}_2$ is a flat locally free sheaf.
Since the quotients and reflexive hulls of generically nef sheaves remain generically nef,
we see that $\mathcal{G}_{2}^{\vee\vee}$ is generically nef. 
This implies that
$$
\mu_{H_1 \cdot H_1 \cdots H_{n-2} }^{\min}(\mathcal{G}_{2}^{\vee\vee}) = \mu_{H_1 \cdot H_1 \cdots H_{n-2} } (\mathcal{G}_{2}^{\vee\vee}) = 0,
$$
and that $\mathcal{G}_{2}^{\vee\vee}$ is $(H_1 \cdot H_1 \cdots H_{n-2})$-semistable.
Thus $\mathcal{G}^{\vee\vee}_{2}$ is a flat locally free sheaf by \cite[Theorem 1.4]{LT18}.
\end{proof}

The following corollary is a direct consequence of Theorem \ref{thm-Langer-Liu-Liu-inequality}. 
Note that the equality \eqref{eq-inequality-c2-nu2} is well known to experts 
(see \cite[Theorem~6.1]{Miy87}, \cite[Proposition~10.12]{K++}).

\begin{cor}[{cf.\,\cite{Miy87, K++}}]
\label{cor-inequality-c2-nu2}
Consider the same situation as in Theorem \ref{thm-Langer-Liu-Liu-inequality}. 
Then, the following inequality holds$:$
\begin{equation}
\label{eq-inequality-c2-nu2}
\widehat{c}_2(\mathcal{E}) H_1 \cdots H_{n-2}  \ge 0.
\end{equation}
Moreover, if equality holds in \eqref{eq-inequality-c2-nu2}, 
then the $(c_1(\mathcal{E}) \cdot H_1 \cdots H_{n-2})$-Harder-Narasimhan filtration of $\mathcal{E}$ is given by
\[
0 \longrightarrow \mathcal{G}_1 := \mathcal{E}_{\max} \longrightarrow \mathcal{E} \longrightarrow \mathcal{G}_2 \longrightarrow 0
\]
and satisfies the following properties$:$
\begin{enumerate}[label=$(\arabic*)$]
\item $\mathcal{G}_1$ is a $\Q$-line bundle with $c_1(\mathcal{G}_1) = c_1(\mathcal{E})$.
\item $\mathcal{G}_2$ is reflexive and satisfies that $c_1(\mathcal{G}_2) = 0$ and $\widehat{\Delta}(\mathcal{G}_{2}^{\vee\vee}) H_1 \cdots H_{n-2} = 0$.
\end{enumerate}
In particular, if $X$ is maximally quasi-\'etale, then $\mathcal{G}_{2}$ is a flat locally free sheaf.
\end{cor}

\begin{proof}
The sheaves $\mathcal{G}_{1}$ and $\mathcal{G}_{2}^{\vee\vee}$ satisfy the
assumptions of Lemma \ref{lem-(co)tangent-extension-of-proj-flat}, 
and thus $\mathcal{G}_{2}$ is reflexive. 
Then, the desired conclusion follows from Theorem \ref{thm-Langer-Liu-Liu-inequality}. 
\end{proof}

\subsection{Case of \texorpdfstring{$\nu(\det \mathcal{E}) \leq 1$}{}}\label{subsec-case2}
In this subsection, we describe the structure of the Harder-Narasimhan filtration in the case where $\nu(\det \mathcal{E}) \leq 1$.

The following theorem generalizes \cite[Lemma 4.5]{Cao13} and \cite[Subsection 6.2]{Ou17}, initially established for smooth varieties, to varieties with klt singularities. 
Note that, as with the equality \eqref{eq-inequality-c2-nu2}, the inequality \eqref{eq-c2-inequality-nu1} is also well known (see \cite[Theorem~6.1]{Miy87} and \cite[Proposition~10.12]{K++}).

\begin{thm}[{cf.\, \cite{Miy87, K++, Cao13, Ou17}}]
\label{thm-c2-inequality-nu1}
Let $H, H_{1}, \ldots, H_{n-2}$ be ample $\Q$-Cartier divisors on a projective klt variety $X$, and let $\mathcal{E}$ be a reflexive sheaf on $X$ satisfying the following conditions. 
\begin{itemize}
\item $\mathcal{E}$ is generically nef. 
\item $\det \mathcal{E}$ is a nef $\Q$-line bundle. 
\item $c_1(\mathcal{E})^2 H_1 \cdots H_{n-2}  = 0$, equivalently, $\nu(\det \mathcal{E}) \leq 1$. 
\end{itemize}
Then, the following inequality holds$:$
\begin{equation}
\label{eq-c2-inequality-nu1}
\widehat{c}_2(\mathcal{E}) H_1 \cdots H_{n-2}  \ge 0.
\end{equation}

Moreover, if equality  holds in \eqref{eq-c2-inequality-nu1}, then there exists a positive rational number $\varepsilon_0$ such that for any rational number $0 < \varepsilon < \varepsilon_0$, the $((c_1(\mathcal{E})+\varepsilon H) \cdot H_1 \cdots H_{n-2})$-Harder-Narasimhan filtration of $\mathcal{E}$ 
\begin{equation}
    \label{eq-HN-filtration-5}
0 =: \mathcal{E}_0 \subsetneq \mathcal{E}_1 \subsetneq \ldots \subsetneq \mathcal{E}_l := \mathcal{E}
\end{equation}
is independent of $\varepsilon$ and  satisfies the following properties$:$
\begin{enumerate}[label=$(\arabic*)$]
\item $\mathcal{G}_i := \mathcal{E}_i / \mathcal{E}_{i-1}$ is a reflexive and numerically projectively flat sheaf.
\item $\det(\mathcal{G}_i)$ is a $\Q$-line bundle and $c_1(\mathcal{G}_i) = \lambda_i c_1(\mathcal{E})$ for some $\lambda_i \in \Q$.
\item The filtration \eqref{eq-HN-filtration-5} is locally analytically split. 
In particular, for any $i$, the following sequence is exact
\begin{align*}
0 \rightarrow \mathcal{G}_i^\vee \rightarrow \mathcal{E}_i^\vee \rightarrow \mathcal{E}_{i-1}^\vee \rightarrow 0.
\end{align*}

\end{enumerate}
\end{thm}

\begin{proof}
Set $H_{\varepsilon} := c_1(\mathcal{E}) + \varepsilon H$ for any $\varepsilon > 0$ and define $\beta := H_1 \cdots H_{n-2}$.
By \cite[Theorem 2.2 (3)]{Miy87}, there exists a positive number $\varepsilon_0 > 0$ such that the $(H_{\varepsilon} \cdot H_1 \cdots H_{n-2})$-Harder-Narasimhan filtration
$$
0 =: \mathcal{E}_0 \subsetneq \mathcal{E}_1 \subsetneq \ldots \subsetneq \mathcal{E}_l := \mathcal{E}
$$
is independent of any positive rational number $\varepsilon < \varepsilon_0$.

Since $\mathcal{G}_{i}^{\vee\vee}$ is $(H_{\varepsilon} \cdot H_1 \cdots H_{n-2})$-semistable, the sheaf $\mathcal{G}_{i}^{\vee\vee}$ satisfies the Bogomolov-Gieseker inequality 
\begin{equation}
\label{eq-Bogomolov-Gieseker-klt-nu1}
\widehat{\Delta}(\mathcal{G}_{i}^{\vee\vee}) \beta \ge 0.
\end{equation}
By the assumption $c_1(\mathcal{E})^2 \beta = 0$, we have
$$
\sum_{i=1}^{l} c_1(\mathcal{G}_i) c_1(\mathcal{E}) \beta = c_1(\mathcal{E})^2 \beta = 0.
$$
Since $\mathcal{E}$ is generically nef, each term $c_1(\mathcal{G}_{i}) c_1(\mathcal{E}) \beta$ on the left-hand side is non-negative.
This implies that $c_1(\mathcal{G}_i) c_1(\mathcal{E}) \beta = 0$.
By applying the Hodge index theorem (see part $(2)$ of Proposition \ref{prop-hodge-index-type}) to $A = (1/r_{i}) c_1(\mathcal{G}_i) - (1/r_{j}) c_1(\mathcal{G}_j)$ and $B = c_1(\mathcal{E})$, we obtain
\begin{equation}
\label{eq-square-negative}
\left( \frac{c_1(\mathcal{G}_i)}{r_i} - \frac{c_1(\mathcal{G}_j)}{r_j} \right)^2 \beta \le 0.
\end{equation}
The desired inequality \eqref{eq-c2-inequality-nu1} follows from
\begin{align}
\begin{split}
\label{eq-Langer-Delta2-c2}
2 \widehat{c}_2(\mathcal{E}) \beta
= \frac{\widehat{\Delta}(\mathcal{E})}{r} \beta
& \underalign{\text{(by Lem. \ref{lem-Langer-Theorem:5.1})}}{\ge} \quad \sum_{i=1}^{l} \frac{\widehat{\Delta}(\mathcal{G}_{i}^{\vee\vee})}{r_i} \beta - \frac{1}{r} \sum_{1 \le i < j \le l} r_i r_j \left( \frac{c_1(\mathcal{G}_i)}{r_i} - \frac{c_1(\mathcal{G}_j)}{r_j} \right)^2 \beta \\
& \underalign{\text{(by \eqref{eq-Bogomolov-Gieseker-klt-nu1})}}{\ge} \quad - \frac{1}{r} \sum_{1 \le i < j \le l} r_i r_j \left( \frac{c_1(\mathcal{G}_i)}{r_i} - \frac{c_1(\mathcal{G}_j)}{r_j} \right)^2 \beta \\
& \underalign{\text{(by \eqref{eq-square-negative})}}{\ge} \quad  0.
\end{split}
\end{align}

For the latter conclusion, assume that $\widehat{c}_2(\mathcal{E}) \beta = 0$. 
We first verify property (2) for $\mathcal{G}_1$. By applying Proposition \ref{prop-hodge-index-type} and Lemma \ref{lem-Q-Cartier-of-Gi} to \eqref{eq-square-negative} (as in the proof of Proposition \ref{prop-Langer-inequality} (2)), we can find rational numbers $\lambda_{j}$ such that
$$
\frac{c_1(\mathcal{G}_1)}{r_1} - \frac{c_1(\mathcal{G}_j)}{r_j} = \lambda_{j} c_1(\mathcal{E}).
$$
Setting $\lambda := r \sum_{j=2}^{l} \frac{r_j}{r_1} \lambda_{j} + 1$, we obtain $c_1(\mathcal{G}_1) = \lambda c_1(\mathcal{E})$. By applying Lemma \ref{lem-Q-Cartier-of-Gi} again, we see that $\det \mathcal{G}_1$ is a $\Q$-line bundle. The same proof applies for $\mathcal{G}_i$, so we omit the details.

Finally, we confirm properties (1) and (3). By \eqref{eq-Bogomolov-Gieseker-klt-nu1}, $\widehat{c}_2(\mathcal{E}) \beta = 0$, and property (2), we deduce that $\widehat{\Delta}(\mathcal{G}_{i}^{\vee\vee}) \beta = 0$, which establishes property (1). Applying Lemma \ref{lem-(co)tangent-extension-of-proj-flat} repeatedly yields property (3).
\end{proof}

\section{Proof of Theorem \ref{thm-main1}}\label{Sec-proof}

\subsection{Case of \texorpdfstring{$\nu(K_{X}) \leq 1$}{}}\label{Ssec-proof}

In this subsection, we first establish Theorem \ref{thm-main1}~(A), 
and then prove Theorem \ref{thm-main1} in the case where $\nu(K_{X}) \leq 1$.

\begin{proof}[Proof of Theorem \ref{thm-main1}~$(A)$ in the case $\nu(K_{X}) \leq 1$]
By \cite[Corollary~6.4]{Miy87} and \cite{CKT21}, the cotangent sheaf $\Omega^{[1]}_X$ is generically nef. 
Therefore, Miyaoka's inequality
\begin{equation*}
    \big( 3\widehat{c}_2(\Omega_{X}^{[1]}) - c_1(\Omega_{X}^{[1]})^2 \big) H_1 \cdots H_{n-2} \geq 0
\end{equation*}
follows from Theorem \ref{thm-c2-inequality-nu1}. 
\end{proof}

Next, we verify Theorem \ref{thm-main1}~(B) in the case $\nu(K_{X})=0$, which is a direct consequence of \cite{LT18}.

\begin{proof}[Proof of Theorem \ref{thm-main1}~$(B)$ in the case $\nu(K_{X})=0$]
Assume that $\nu(K_{X})=0$. 
By \cite[Theorem 1.2]{LT18}, we conclude that $X$ is a quasi-\'etale quotient of an abelian variety, which completes the proof.
\end{proof}

Hereafter, in this subsection, we focus on the case $\nu(K_{X})=1$. 
Our first goal is to prove that $K_{X}$ is semi-ample (see Theorem \ref{thm-Abundance-theorem-nu1}).
Our argument is strongly influenced by the approach of Lazi\'c-Peternell \cite{LP18} to the non-vanishing problem in the case $\nu(K_X)=1$.

\begin{thm}\label{thm-Abundance-theorem-nu1}
Let $X$ be a projective klt variety as in Theorem \ref{thm-main1}. 
Assume that $\nu(K_{X}) = 1$.  
Then, the canonical divisor $K_{X}$ is semi-ample.
\end{thm}

\begin{proof}
The proof of Theorem \ref{thm-Abundance-theorem-nu1} involves an in-depth analysis of the Harder-Narasimhan filtration of the cotangent sheaf $\Omega^{[1]}_X$, which was studied in Section \ref{section:3}.
Since $\nu(K_{X})=1$, our assumption on Miyaoka's equality is equivalent to 
$$
\widehat{c}_2(\Omega_{X}^{[1]})  H_1 \cdots H_{n-2} = 0.
$$
It is enough to prove that $K_{X}$ is semi-ample after replacing $X$ with its finite quasi-\'etale cover. 
Thus, we may assume that $X$ is maximally quasi-\'etale (see \cite[Theorem~1.14]{GKP16b}). 

Lemma \ref{lem-Harder-Narasimhan-case-v=1}, which is essential for the proof of Theorem \ref{thm-Abundance-theorem-nu1}, is an extension of the main result of \cite{IM22}.

\begin{lem}\label{lem-Harder-Narasimhan-case-v=1}
Let $X$ be a projective klt variety as in Theorem \ref{thm-main1}. 
Assume that $\nu(K_{X}) = 1$ and that $X$ is maximally quasi-\'etale. 
Then, the following statements hold$:$
\begin{enumerate}[label=$(\arabic*)$]
\item $X$ has  finite quotient singularities and is $\mathbb{Q}$-factorial. 
\item No resolution $\pi \colon \widetilde{X} \to X$ is of special type in the sense of Campana.  
\end{enumerate}
\end{lem}

\begin{proof}[Proof of Lemma \ref{lem-Harder-Narasimhan-case-v=1}]
By \cite[Corollary~6.4]{Miy87} and \cite{CKT21}, the cotangent sheaf $\Omega^{[1]}_X$ is generically nef. 
Thus, by applying Theorem \ref{thm-c2-inequality-nu1} to $\Omega^{[1]}_X$, we deduce that the $\alpha_{\varepsilon}$-Harder-Narasimhan filtration of $\Omega^{[1]}_X$
\begin{align}\label{HN-nu=1}
0 =: \mathcal{E}_0 \subsetneq \mathcal{E}_1 \subsetneq \ldots \subsetneq \mathcal{E}_l := \Omega^{[1]}_X
\end{align}
satisfies the properties stated in Theorem \ref{thm-c2-inequality-nu1}, 
where $\alpha_{\varepsilon} := (K_{X} + \varepsilon H) \cdot H_1 \cdots H_{n-2}$ for some ample divisor $H$ and sufficiently small $\varepsilon > 0$. In particular, each sheaf $\mathcal{G}_i$ is semistable. 

Fix a point $x \in X$. By Lemma \ref{lem-Proj-flat-bundles-are-Q-vectorbundles}, there exist an analytic neighbourhood $x \in U \subset X$ and a finite quasi-\'etale Galois cover $\pi\colon V \rightarrow U$ such that $\pi^{[*]}\mathcal{G}_i$ is locally free for every $i$. 
By Theorem \ref{thm-c2-inequality-nu1}, the filtration \eqref{HN-nu=1} is locally analytically split. 
Then, after shrinking $U$, we have
$$
\Omega_V^{[1]} \cong \pi^{[*]}\Omega_U^{[1]} \cong \bigoplus_i \pi^{[*]}\mathcal{G}_i,
$$
which is locally free. 
The variety $V$ has klt singularities since $V$ is a finite quasi-\'etale cover of the klt variety $U$. 
By \cite[Theorem~6.1]{GKKP11}, we conclude that $V$ is smooth, which implies that $U$ (and thus $X$) has only finite quotient singularities. 
In particular, by \cite[Proposition~5.15]{KM98},  we see that the variety $X$ is $\Q$-factorial.

To prove $(2)$, suppose that there exists a resolution $\pi\colon \widetilde{X} \rightarrow X$ of the singularities of $X$ such that $\widetilde{X}$ is of special type in the sense of Campana. 
By \cite[Theorem~A]{PRT22}, to reach a contradiction it suffices to show that there exists a sub-line bundle $\mathcal{L} \subset \Omega_{\widetilde{X}}^1$ with $\nu(\mathcal{L}) = 1$. 

To construct such a sub-line bundle, we focus on the first piece $\mathcal{E}_1$ of the $\alpha_{\varepsilon}$-Harder-Narasimhan filtration \eqref{HN-nu=1}. The sheaf $\mathcal{E}_1$ is $\alpha_{\varepsilon}$-semistable with 
$\mu_{\alpha_{\varepsilon}}(\mathcal{E}_1) \geq \mu_{\alpha_{\varepsilon}}(\Omega_X^{[1]}) > 0$. 
Moreover $\mathcal{E}_1$ is numerically projectively flat and satisfies 
$$
\widehat{c}_2(\mathcal{E}_1) = c_1(\mathcal{E}_1)^2 = 0.
$$
Hence, Corollary \ref{cor-locfree} implies that $\mathcal{E}_{1}$ is locally free and projectively flat.

Since $\pi_1(\widetilde{X}) \cong \pi_1(X)$ by \cite{Tak03}, 
we deduce that the image of any linear representation of $\pi_1(X)$ is virtually abelian by \cite[Theorem~7.8]{Cam04}.
Thus, by \cite[Proposition~3.3]{GKP21}, after passing to a finite \'etale cover of $X$ and $\widetilde{X}$, 
there exists a line bundle $\mathcal{A}$ on $X$ such that $\mathcal{E}_1 \otimes \mathcal{A}^{\vee} \cong \mathcal{E}'$ is a flat locally free sheaf admitting a filtration
$$
0 \subsetneq \mathcal{O}_X \cong \mathcal{E}'_1 \subsetneq \ldots \subsetneq \mathcal{E}'_k = \mathcal{E}',
$$
where each $\mathcal{E}'_i/\mathcal{E}'_{i-1}$ is a flat line bundle. 
The line bundle $\mathcal{L} := \mathcal{E}'_1 \otimes \mathcal{A}$ is a sub-line bundle of $\Omega_{X}^{[1]}$ with $\nu(c_1(\mathcal{L})) = \nu(c_1(\mathcal{E}_1)) = 1$. The pull-back $\pi^*\mathcal{L} \hookrightarrow \pi^{[*]}\Omega_X^{[1]} \hookrightarrow \Omega_{\widetilde{X}}^1$ is the desired line bundle, yielding a contradiction by \cite[Theorem~A]{PRT22}. 
\end{proof}
We now return to the proof of Theorem \ref{thm-Abundance-theorem-nu1}. Let $\pi \colon \widetilde{X} \rightarrow X$ be a log resolution of the singularities of $X$. 
Since $X$ has klt singularities, there exist $\pi$-exceptional effective divisors $E$ and $G$ such that
\begin{equation}
\label{eq-klt}
\pi^{*}K_{X} + E \sim_{\Q} K_{\widetilde{X}} + G \quad \text{and} \quad (\widetilde{X}, G) \text{ is a log smooth klt pair.}
\end{equation}
Since $\widetilde{X}$ is not of special type by Lemma \ref{lem-Harder-Narasimhan-case-v=1}, there exists a Bogomolov sheaf $\mathcal{L} \subset \Omega_{\widetilde{X}}^{p}$ for some $p \geq 1$ (i.e.,\,a sub-line bundle $\mathcal{L} \subset \Omega_{\widetilde{X}}^{p}$ with $\kappa(\mathcal{L}) = p$). By taking the saturation, we may assume that $\mathcal{L}$ is saturated in $\Omega^{p}_{\widetilde{X}}(\log \lceil G \rceil)$.

Consider the exact sequence
\begin{equation}
\label{eq-exact-sequence-CP}
0 \longrightarrow \mathcal{L} \longrightarrow \Omega^{p}_{\widetilde{X}}(\log \lceil G \rceil) 
\longrightarrow \mathcal{Q} := \Omega^{p}_{\widetilde{X}}(\log \lceil G \rceil)/\mathcal{L} 
\longrightarrow 0.
\end{equation}
Since $K_{\widetilde{X}} + \lceil G \rceil$ is pseudo-effective by \eqref{eq-klt}, we conclude that $\det \mathcal{Q}$ is pseudo-effective by \cite[Theorem~1.3]{CP19} and \cite{BDPP13}.
Combining \eqref{eq-klt} and \eqref{eq-exact-sequence-CP}, we obtain
$$
\pi^{*}K_{X} + E + \lceil G \rceil - G \sim_{\Q} K_{\widetilde{X}} + \lceil G \rceil \sim_{\Q} \mathcal{L} + \det \mathcal{Q}.
$$
Hence, we deduce that 
$$
K_{X} \sim_{\Q} \pi_{[*]} \mathcal{L} + \pi_{[*]} \det \mathcal{Q},
$$
by noting that the reflexive hull $\pi_{[*]}(\bullet)$ of the direct image of a line bundle is a $\mathbb{Q}$-line bundle by the $\mathbb{Q}$-factoriality of $X$. Since $\pi_{[*]} \det \mathcal{Q}$ is pseudo-effective, 
we obtain
$$
\kappa(K_{X}) \geq \kappa(\pi_{[*]} \mathcal{L}) = \kappa(\mathcal{L}) = p \ge 1 
$$
by \cite[Corollary~6.2]{LP18}, which confirms that $\kappa(K_{X}) = \nu(K_{X}) = 1$. 
This proves that $K_{X}$ is semi-ample.
\end{proof}

At the end of this subsection, after recalling the definition of $C^\infty$-orbifold fibre bundles, 
we establish Theorem \ref{thm-Abundance-theorem-nu1-str}, thereby completing the proof of Theorem \ref{thm-main1} in the case $\nu(K_{X}) = 1$.

\begin{thm}\label{thm-Abundance-theorem-nu1-str}
Let $X$ be a projective klt variety as in Theorem \ref{thm-main1}. 
Assume that $\nu(K_{X}) = 1$. 
Then, after replacing $X$ with a finite quasi-\'etale cover, 
the variety $X$ is smooth and admits an abelian group scheme $X \rightarrow C$ over a curve $C$ of general type.
\end{thm}
\begin{defn}[$($stratified$)$ $C^\infty$-orbifold fibre bundles]\label{def-infty}
Let $f\colon X \rightarrow C$ be a fibration from a normal projective variety $X$ onto a smooth curve $C$.
\begin{enumerate}
    \item[(1)] Let $f^{-1}(t) = \sum_{i=1}^\ell a_i F_i$ be the irreducible decomposition of the divisor $f^{-1}(t)$ over a point $t \in C$.
    The integer $m := \gcd(a_1, \ldots, a_\ell)$ is called the \emph{multiplicity} over $t \in C$.
    
    \item[(2)] The divisor $D_f$ on $C$ defined by 
    $$
    D_f := \sum_{i=1}^s \left(1 - \frac{1}{m_i}\right) t_i 
    $$
    is called the \emph{orbifold branch divisor of $f$}. 
    Here $m_i$ denotes the multiplicity over $t_i$, and 
$t_i$ runs through all points with multiplicity $m_i \geq 2$.

    \item[(3)] The fibration $f\colon X \rightarrow (C, D_f)$ is called a 
    \emph{$($stratified$)$ $C^\infty$-orbifold fibre bundle} 
    if, for any $t \in C$, there exists an analytic open disc $\Delta \subset C$ containing $t$ with the following property: 
    Let $m$ denote the multiplicity over $t$, and consider the normalization $\widehat{V}$ of the fibre product of 
    $f \colon V := f^{-1}(\Delta) \to \Delta$ and the map
    $p\colon \Delta \rightarrow \Delta,\ z \mapsto z^m$: 
    \[
    \xymatrix{
    \widehat{V} \ar[r]^{\widehat{p}} \ar[d]_{\widehat{f}} & V \ar[d]^f \\
    \Delta \ar[r]^{p} & \Delta, 
    }
    \] 
    Then, there exists a $C^\infty$-isomorphism 
    $$
    \widehat{V} \cong_{C^\infty} \Delta \times F
    $$ 
    over $\Delta$ as $C^\infty$-(stratified) spaces, where $F$ is a general fibre of $f$ 
    (see \cite[Section I.1.5]{GM88} for $C^\infty$-stratified spaces).
    Note that $\widehat{p} \colon \widehat{V} \rightarrow V$ is a finite quasi-\'etale cover 
    since it is \'etale over the smooth locus of $V$, 
   by a straightforward local computation (see \cite[Proof of Lemma 2.2]{Hor07}).
\end{enumerate}
\end{defn}

\begin{prop}\label{prop-orbifold-fibre-bundle-structure}
Under the same assumptions as in Theorem \ref{thm-Abundance-theorem-nu1-str}, 
we consider the Iitaka fibration 
$$
f\colon X \rightarrow (C, D)
$$ 
onto a smooth curve $C$ with orbifold branch divisor $D := D_{f}$. 
Then the fibration $f\colon X \rightarrow (C, D)$ is a 
$($stratified$)$ $C^\infty$-orbifold fibre bundle.
\end{prop}

\begin{proof}
We now examine the filtration \eqref{HN-nu=1} in detail:
\begin{equation*}
0 =: \mathcal{E}_0 \subsetneq \mathcal{E}_1 \subsetneq \ldots \subsetneq \mathcal{E}_\ell := \Omega^{[1]}_X.
\end{equation*}
By replacing this filtration with a Jordan-H\"older refinement, 
we may assume that the sheaves $\mathcal{G}_i := \mathcal{E}_i / \mathcal{E}_{i-1}$ are $\alpha_\varepsilon$-stable. 
Since the sheaves $\mathcal{G}_i$ are numerically projectively flat, 
there exist associated irreducible representations 
$$
\rho_i\colon \pi_1(X) \longrightarrow \mathbb{P}\mathrm{GL}(r_i,\C),
$$
where $r_i := \rk \mathcal{G}_i$.
Moreover, since $c_1(\mathcal{G}_i)$ is numerically proportional to $K_X$, 
there exist integers $a_i, b_i \in \mathbb{Z}_{+}$ and numerically trivial line bundles $\mathcal{M}_i$ on $X$
such that 
\begin{align}
    \det(\mathcal{G}_i)^{[\otimes] b_i}= \mathcal{M}_i \otimes \mathcal{O}_X(a_i K_X). 
    \label{eq-fibreBundleN=1}
\end{align} 
Let 
$$
\sigma_i\colon \pi_1(X) \longrightarrow \GL(1,\C)
$$
be the irreducible representations associated with $\mathcal{M}_i$. 
Furthermore, since $f\colon X \rightarrow (C, D)$ is the Iitaka fibration, 
we may assume that there exists an ample divisor $H_i$ on $C$ such that 
\begin{align}
    \mathcal{O}_X(a_i K_X) \cong f^*\mathcal{O}_C(H_i). 
    \label{eq-fibreBundleN=1_2}
\end{align}

Fix $t \in C$ and let $m \geq 1$ be the multiplicity over $t$. 
Consider the normalization $\widehat{V}$ of the fibre product as in Definition \ref{def-infty}. 
Hereafter, we divide the proof into several steps and show that 
$\widehat{f}\colon \widehat{V} \to \Delta$ is a (stratified) $C^\infty$-fibre bundle.

\vspace{0.5\baselineskip}
\emph{Step 1: The images of $\rho_i|_{\pi_1(\widehat{V})}$ and $\sigma_i|_{\pi_1(\widehat{V})}$ are finite.}
\vspace{0.5\baselineskip}

Set $\rho_i|_{\pi_1(\widehat{V})} := \rho_i \circ (j_{*} \circ \widehat{p}_{*})$ and 
$\sigma_i|_{\pi_1(\widehat{V})} := \sigma_i \circ (j_{*} \circ \widehat{p}_{*})$, 
where $j_{*} \circ \widehat{p}_{*}$ is the homomorphism
$$
j_{*} \circ \widehat{p}_{*} \colon \pi_1(\widehat{V}) \xrightarrow{\widehat{p}_{*}} \pi_1(V) 
 \xrightarrow{j_{*}} \pi_{1}(X),
$$
induced by $\widehat{p} \colon \widehat{V} \to V$ and 
the natural inclusion $j \colon V \hookrightarrow X$. 
In this step, we show that the images of $\rho_i|_{\pi_1(\widehat{V})}$ and $\sigma_i|_{\pi_1(\widehat{V})}$ are finite.

Let $F$ be a general fibre of $\widehat{f}$. 
The fibre $F$ is a klt variety satisfying $K_F \sim_\Q 0$ and $\widehat{c}_2(F) = 0$. 
Thus, by \cite[Theorem 1.2]{LT18}, there exists a finite quasi-\'etale cover $\pi\colon A \rightarrow F$ such that $A$ is an abelian variety. 
We identify $\pi\colon A \rightarrow F$ 
with the composition $A \xrightarrow{\pi} F \hookrightarrow X$ 
and consider the reflexive pull-back of the filtration \eqref{HN-nu=1} via $\pi^{[*]}$. 
Since $A$ is an abelian variety, $\pi^{[*]}\Omega^{[1]}_X$ admits a filtration by trivial bundles on $A$. 
On the other hand, the sheaf $\pi^{[*]}\Omega^{[1]}_X$ also admits a filtration 
whose graded pieces are $\pi^{[*]}\mathcal{G}_i$, which are polystable sheaves of slope zero. 
Hence, we deduce $\pi^{[*]}\mathcal{G}_i \cong \mathcal{O}_A^{\oplus r_i}$.

Therefore, the induced homomorphism 
$$
\pi_{1}(A) \longrightarrow \pi_{1}(F) \longrightarrow \pi_{1}(X) \xrightarrow{\rho_{i}} \mathbb{P}\mathrm{GL}(r_i,\C)
$$
is trivial. 
Similarly, by \eqref{eq-fibreBundleN=1} and \eqref{eq-fibreBundleN=1_2}, we conclude that 
$$
\pi_{1}(A) \longrightarrow \pi_{1}(F) \longrightarrow \pi_{1}(X) \xrightarrow{\sigma_{i}} \mathrm{GL}(1,\C)
$$
is trivial. 
On the other hand, the sequence
\begin{equation}\label{eq-homotopy-right-exact-sequence}
\widehat{\pi}_1(F) \longrightarrow \widehat{\pi}_1(\widehat{V}) \longrightarrow \widehat{\pi}_1(\Delta) = \{1\} 
\end{equation}
is exact (see \cite[\href{https://stacks.math.columbia.edu/tag/0C0J}{Tag 0C0J}]{stacks-project} for the proof applies verbatim to morphisms of analytic varieties). 
Thus, the image 
$$
H := \operatorname{Im}\Big(\widehat{\pi}_1(A) \longrightarrow \widehat{\pi}_1(F) \longrightarrow \widehat{\pi}_1(\widehat{V})\Big)
$$
has finite index in $\widehat{\pi}_1(\widehat{V})$.
This implies that the homomorphisms
$$
\widehat{\rho_i}\colon \widehat{\pi}_1(\widehat{V}) \longrightarrow \widehat{\operatorname{Im}(\rho_i)} \quad \text{ and }\quad 
\widehat{\sigma_i}\colon \widehat{\pi}_1(\widehat{V}) \longrightarrow \widehat{\operatorname{Im}(\sigma_i)}
$$
have finite image. 
By Malcev's theorem \cite[Theorem~4.2]{Weh73}, 
the images of $\rho_i|_{\pi_1(\widehat{V})}$ and $\sigma_i|_{\pi_1(\widehat{V})}$ are finite.

\vspace{0.5\baselineskip}
\emph{Step 2: $\widehat{V}$ is smooth up to finite quasi-\'etale covers.}
\vspace{0.5\baselineskip}

In this step, we show that there exists a finite quasi-\'etale Galois cover 
$\psi \colon \widehat{W} \rightarrow \widehat{V}$ such that $\widehat{W}$ is smooth 
and $h := \widehat{f} \circ \psi \colon \widehat{W} \rightarrow \Delta$ is an abelian group scheme.

Let $\psi \colon \widehat{W} \rightarrow \widehat{V}$ be the finite \'etale cover corresponding to 
the intersection of the kernels of all $\rho_i|_{\pi_{1}(\widehat{V})}$ and $\sigma_i|_{\pi_{1}(\widehat{V})}$.  
By construction, there exist reflexive coherent sheaves $\mathcal{L}_i$ of rank one on $\widehat{W}$
such that $\mathcal{G}_i|_{\widehat{W}} \cong \mathcal{L}_i^{\oplus r_i}$. 
Moreover, by \eqref{eq-fibreBundleN=1} and \eqref{eq-fibreBundleN=1_2}, 
we have $\mathcal{L}_i^{[\otimes] c_i} \cong \mathcal{O}_{\widehat{W}}$, 
where $c_i := a_i \cdot b_i \cdot r_i$. 
Thus, after replacing $\widehat{W}$ with the associated index-one cover, 
we may assume that each $\mathcal{L}_i$ is trivial. 
It follows that $\Omega^{[1]}_{\widehat{W}}$ is locally free, since it is a successive extension of $\mathcal{L}_i$ 
by the filtration \eqref{HN-nu=1}. 
Hence, by \cite[Theorem~6.1]{GKKP11}, we conclude that $\widehat{W}$ is smooth.

By taking a further finite \'etale cover, we may assume that $\psi \colon \widehat{W} \rightarrow \widehat{V}$ is a Galois cover. 
By construction, the fibration $\widehat{W} \rightarrow \Delta$ has multiplicity one. 
Furthermore, the above argument shows that 
the cotangent bundle $\Omega^{[1]}_{\widehat{W}}$ is $(\widehat{f}\circ \psi)$-numerically flat.  
Thus, by \cite[Lemma~2.1(1)]{Hor13}, we conclude that $\widehat{W} \rightarrow \Delta$ is a smooth fibration. 
Since $\widehat{W}$ is connected, its fibres are also connected. 
After possibly replacing $\widehat{W}$ with a further finite \'etale cover, 
we may assume that the fibres of $\widehat{W} \rightarrow \Delta$ are abelian varieties 
(see \cite[Lemma~2.1(2)]{Hor13}). 
This fibration clearly admits a section, and the claim follows.

\vspace{0.5\baselineskip}
\emph{Step 3: The desired conclusion.}
\vspace{0.5\baselineskip}

In this step, we show that $\widehat{V} \rightarrow \Delta$ is a $C^\infty$-fibre bundle.
Let $A$ be a general fibre of $h\colon \widehat{W}\rightarrow \Delta$. 
Then, by Step 2, we have $\widehat{W} \cong_{C^\infty} A \times \Delta$. 
Under this identification, the action of the covering group $G$ of $\widehat{W} \rightarrow \widehat{V}$ is given by
$$
G \curvearrowright \widehat{W} \cong_{C^\infty} A \times \Delta, \quad g\cdot (z, t) = \big(\theta_g(z, t), t\big),
$$
where $z$ is a local coordinate on $A$ and $G$ is the Galois group of $\psi \colon \widehat{W} \rightarrow \widehat{V}$. 
Note that $\theta_g(z, t)$ depends smoothly on $z$ and $t$. 
By \cite[Theorem~2.4.B']{GP91}, there exists (possibly after changing the diffeomorphism) a $C^\infty$-isomorphism $\widehat{W} \cong_{C^\infty} A \times \Delta$ such that $\theta_g$ is independent of $t$, which completes the proof.
\end{proof}

Our next goal is to show that we can eliminate the multiplicities, i.e.,\,we find a finite orbi-\'etale cover 
$(\widehat{C}, \widehat{D}) \rightarrow (C, D)$ such that $\widehat{D} = 0$. 
This is a priori unclear in the case $C \cong \mathbb{P}^{1}$, which is why we need the following proposition:

\begin{prop}   \label{lem-abelian-fibrations-over-P1}
Under the same assumptions as in Proposition \ref{prop-orbifold-fibre-bundle-structure}, 
we further assume that $C$ is the smooth rational curve $\mathbb{P}^{1}$. 
Then, there exists a finite orbifold-\'etale cover $(\widehat{C}, 0) \rightarrow (\mathbb{P}^1, D)$ 
such that $\widehat{C}$ is a smooth curve and 
the fibre product $\widehat{X} :=\widehat{C}\times_{C} X$ is isomorphic to the product $\widehat{C} \times F$.
\end{prop}
\begin{proof}
Note that $f\colon X \rightarrow (C, D)$ is a $C^\infty$-orbifold fibre bundle by Proposition \ref{prop-orbifold-fibre-bundle-structure}.

\vspace{0.5\baselineskip}
\emph{Step 1: $f\colon X \rightarrow (C, D)$ is a holomorphic orbifold fibre bundle.}
\vspace{0.5\baselineskip}

Fix a point $t_0 \in \mathbb{P}^1$, 
we set $U := \mathbb{P}^1 \setminus \{t_0\} \cong \mathbb{A}^1$ and $V := f^{-1}(U)$. 
Consider the finite orbifold-\'etale cover 
$$
p \colon (\widehat{U} := \mathbb{A}^1, 0) \longrightarrow (U, D), \quad  
z \longmapsto (z - t_i)^{m_i}, 
$$ 
where $\{t_1, \ldots, t_s \} \subset C$ is the set of points with multiplicities $m_i \geq 2$.
The base change $\widehat{f}\colon \widehat{V} := \widehat{U} \times_{U} X \rightarrow \widehat{U}$ is still a $C^\infty$-fibre bundle. 
We now show that all fibres of $\widehat{f}$ are isomorphic, which implies that $f\colon X \rightarrow (C, D)$ is a holomorphic orbifold fibre bundle.
Since $\widehat{U} = \mathbb{A}^1$ is contractible, we have $\pi_1(\widehat{V}) \cong \pi_1(F)$. 
In particular, there exists a finite quasi-\'etale Galois cover $\widehat{W} \rightarrow \widehat{V}$ with Galois group $G$
such that the fibres of $\widehat{W} \rightarrow \widehat{U} \cong \mathbb{A}^1$ are abelian varieties.
Hence $\widehat{W}$ is isomorphic to the product $\widehat{U} \times A$ 
(see, for example, \cite[Proof of Prop.~3.12(vii)]{DPS94}).
By \cite[Theorem~2.4]{GP91}, 
the homomorphism $G \rightarrow \operatorname{Aut}(A)$ is rigid (i.e.,\,the action of $G$ on all fibres is the same). 
This shows that 
\begin{equation*}
\widehat{V} \cong \widehat{W}/G \cong \widehat{U} \times (A/G) = \widehat{U} \times F,
\end{equation*}
which is a holomorphic fibre bundle.

\vspace{0.5\baselineskip}
\emph{Step 2: $\widehat{X}$ is isomorphic to the product $F \times \widehat{C}$ after finite orbifold-\'etale covers.}
\vspace{0.5\baselineskip}

Fix a very ample line bundle $\mathcal{L}$ on $X$. 
Then the structure group of $f$ is reduced to $\operatorname{Aut}(F, \mathcal{L}|_F) \subset \operatorname{Aut}(F)$. 
Now, $\operatorname{Aut}(F, \mathcal{L}|_F)$ is a linear algebraic group (see, for example, \cite[Theorem~2.16]{Bri18}), 
while $\operatorname{Aut}^0(F)$ is an abelian variety since $\kappa(F) \geq 0$ 
(see, for example, \cite[Proposition~4.6]{Amb05}). 
This implies that $\operatorname{Aut}(F, \mathcal{L}|_F)$ is finite.

In particular, the variety $X$ is defined by a representation $\rho\colon \pi_1(C, D) \rightarrow \operatorname{Aut}(F, \mathcal{L}|_F)$ of the orbifold fundamental group. 
Let $(\widehat{C}, \widehat{D}) \rightarrow (C, D)$ be the finite orbifold-\'etale cover corresponding to $\ker(\rho)$. 
Then, by construction, we obtain $\widehat{X} \cong \widehat{C} \times F$, 
and moreover $\widehat{D} = 0$.
\end{proof}

Finally, we make the following elementary observation:

\begin{lem}   \label{prop-quotients-abelian-fundamental-group}
Let $A$ be an abelian variety of dimension $n$, and 
$\pi\colon A \rightarrow F := A/G$ be a finite Galois quasi-\'etale cover with Galois group $G$. 
If $\pi_1(F^{\reg})$ is an abelian group $($possibly with torsion$)$ of rank $2n$, then $F$ is smooth.
\end{lem}
\begin{proof}
Consider the natural isomorphisms
\begin{equation*}
H^0(A, \C)^G \cong H^0(\pi^{-1}(F^{\reg}), \C)^G \cong H^0(F^{\reg}, \C).
\end{equation*}
By assumption, both $H^0(A, \C)$ and $H^0(F^{\reg}, \C)$ are $\C$-vector spaces of the same dimension. 
It follows that the action of $G$ on $H^0(A, \C)$ is trivial. 
Hence $G$ acts by translations on $A$, which implies that the action has no fixed points, and therefore $F$ is smooth.
\end{proof}

\begin{proof}[Proof of Theorem \ref{thm-Abundance-theorem-nu1-str}]
By Proposition \ref{prop-orbifold-fibre-bundle-structure}, up to finite quasi-\'etale covers, 
the Iitaka fibration $f\colon X \rightarrow (C, D)$ is a (stratified) $C^\infty$-orbifold fibre bundle. 
In the case $C \cong \mathbb{P}^{1}$, the proof follows from Proposition \ref{lem-abelian-fibrations-over-P1}, 
so we may assume that $C$ is a curve of genus $g(C) \geq 1$.
In particular $K_C + D$ is either numerically trivial or ample. 
Hence, there exists a finite orbi-\'etale cover $(\widehat{C}, \widehat{D}) \rightarrow (C, D)$ with $\widehat{D} = 0$ (see, for example, \cite{CGG23}). 
After replacing $X$ by the induced quasi-\'etale cover $\widehat{X} \rightarrow X$, 
we may assume that $f\colon X \rightarrow C$ is a (stratified) $C^\infty$-fibre bundle. 
Moreover, since $g(C) \geq 1$, we have $\pi_2(C) = 1$. 
Thus, we obtain the following long exact sequence of homotopy groups:
\begin{equation*}
1 = \pi_2(C) \longrightarrow \pi_1(F_{\reg}) \longrightarrow \pi_1(X_{\reg}) \longrightarrow \pi_1(C) \longrightarrow 1.
\end{equation*}
Since $F$ is a finite quasi-\'etale quotient of an abelian variety, say $F = A/G$, the group $\pi_1(F_{\reg})$ contains a free abelian subgroup of rank $2(n-1)$ of finite index. 
By \cite[Corollary~6.4.3]{Kol93}, after replacing $X$ with a finite quasi-\'etale cover, 
we may assume that $\pi_1(F^{\reg})$ is abelian of rank $2(n-1)$. 
By Lemma \ref{prop-quotients-abelian-fundamental-group}, 
the variety $F$ (and hence $X$) is smooth. 
Finally, by the abundance theorem of \cite{IM22}, 
the structure of abelian group schemes follows from \cite{Hor13}.
\end{proof}

\subsection{Case of \texorpdfstring{$\nu(K_{X}) \geq 2$}{}}\label{Ssec-proof2}

This subsection is devoted to proving Theorem \ref{thm-main1} in the case $\nu(K_X) \geq 2$. 
To this end, we study certain inequalities for the second $\Q$-Chern class of the cotangent sheaf $\Omega_{X}^{[1]}$, 
which have already been established in the case where $X$ is smooth in codimension two 
(see \cite[Chapter~7]{Miy87}, \cite[Section~5]{Lan02}, and \cite[Theorem~7.2]{RT16}). 
Our contribution is to extend these inequalities to varieties with klt singularities 
and to examine the conditions under which equality is attained. 
To this end, we make use of the theory of Higgs sheaves. 
We now consider the following setup.

\begin{setup}
\label{setup-minimal-variety}
Let $X$ be a projective klt variety of dimension $n \geq 2$, and 
let $H_1, \ldots, H_{n-2}$ be ample $\Q$-Cartier divisors on $X$.
Assume that $K_X$ is nef and $\nu(K_X) \geq 2$. 
As in Section~\ref{section:3}, 
consider the $(K_X \cdot H_1 \cdots H_{n-2})$-Harder-Narasimhan filtration 
of the cotangent sheaf $\Omega_{X}^{[1]}$:
$$
0 =: \mathcal{E}_0 \subset \mathcal{E}_1 \subset \ldots \subset \mathcal{E}_l := \Omega_{X}^{[1]}. 
$$
We define the constants $\mu_{i}$, $\mu$, and $d$ by 
$$
\mu_{i} := \mu_{K_X \cdot H_1 \cdots H_{n-2}}(\mathcal{G}_i), \quad
\mu := \mu_{K_X \cdot H_1 \cdots H_{n-2}}(\Omega_{X}^{[1]}), \quad
d := c_1(\Omega_{X}^{[1]})^{2} \cdot H_1 \cdots H_{n-2}. 
$$
\end{setup}

We investigate inequalities for the second $\Q$-Chern class $\widehat{c}_{2}(\Omega_{X}^{[1]})$, 
dividing our situation into the cases $r_1 \geq 2$ and $r_1 = 1$. 
The following theorems (Theorems \ref{thm-Miyaoka-inequality-1} and \ref{thm-Miyaoka-inequality-2}) 
have been established for smooth projective varieties $X$ in \cite[Theorem~5.5]{Lan02}, 
but here we generalize them to varieties with klt singularities and examine the cases where equality holds.

\begin{thm} [{cf.\,\cite[Theorem 5.5]{Lan02}}]
\label{thm-Miyaoka-inequality-1}
Consider Setup \ref{setup-minimal-variety} and assume that $r_1 \geq 2$.
Then the following inequality holds:
\begin{equation}
\label{eq-Miyaoka-inequality-1}
\widehat{c}_2(\Omega_{X}^{[1]}) H_1 \cdots H_{n-2} 
\geq \frac{r_1}{2(r_1+1)}\, c_1(\Omega_{X}^{[1]})^{2} H_1 \cdots H_{n-2}.
\end{equation}

Moreover, if $X$ is maximally quasi-\'etale and equality holds in \eqref{eq-Miyaoka-inequality-1}, 
then $X$ is smooth and the $(K_X \cdot H_1 \cdots H_{n-2})$-Harder-Narasimhan filtration of $\Omega_{X}^{1}$ is given by
\begin{equation*}
0 \longrightarrow \mathcal{G}_1 \longrightarrow \Omega_{X}^{1} \longrightarrow \mathcal{G}_2 \longrightarrow 0,
\end{equation*}
and satisfies the following properties$:$
\begin{enumerate}[label=$(\arabic*)$]
\item $\mathcal{G}_1$ is a locally free sheaf of rank $r_1$ such that
\[
\bigl( 2(r_1+1)\,\widehat{c}_2(\mathcal{G}_1) - r_1\, c_1(\mathcal{G}_1)^{2} \bigr) H_1 \cdots H_{n-2} = 0.
\]
\item $\mathcal{G}_{2}$ is flat and locally free.
\end{enumerate}
\end{thm}
\begin{proof}
Set $\polar := H_1 \cdots H_{n-2}$. 
We first consider the second $\Q$-Chern class of $\mathcal{G}_1$.
By applying Proposition \ref{prop-Higgs-1} to $\mathcal{G}_1 \subset \Omega_{X}^{[1]}$, 
we obtain the $(K_X \cdot H_{1} \cdots H_{n-2})$-stable Higgs sheaf $(\mathcal{G}_1 \oplus \mathcal{O}_{X}, \theta)$, 
where the Higgs field $\theta$ is defined as in Proposition \ref{prop-Higgs-1}. 
Then, the Bogomolov-Gieseker inequality in Proposition \ref{prop-Higgs-2} yields 
\begin{equation}
\label{eq-Delta2-Higgs-G1}
\widehat{c}_2(\mathcal{G}_1)\polar
\geq \frac{r_1}{2(r_1+1)}\, c_1(\mathcal{G}_1)^2 \polar.
\end{equation}

Meanwhile, we have 
\begin{equation}
\label{eq-minimal-variety-generically-nef}
 r_1 \mu_1 \leq n\mu = d \quad  \text{and} \quad  \mu_{l} \geq 0 
\end{equation}
by Miyaoka's semipositivity theorem 
(see \cite[Corollary~6.4]{Miy87}, \cite[Theorem~1.3]{CP19}, and \cite[Theorem~5.5]{CKT21}).
The Hodge index theorem (see Proposition \ref{prop-hodge-index-type}) implies that 
\begin{equation}
\label{eq-Miyaoka-2}
\sum_{i=2}^{l} \frac{1}{r_i} \widehat{c}_1(\mathcal{G}_{i}^{\vee\vee})^{2}\polar
\underset{\text{(by Prop.~\ref{prop-hodge-index-type})}}{\leq}  
\frac{1}{d} \sum_{i=2}^{l} r_i \mu_{i}^{2} 
\leq  \frac{\mu_1}{d} \sum_{i=2}^{l} r_i \mu_{i}
= \mu_1 - \frac{r_1 \mu_{1}^{2}}{d}.
\end{equation}
As in \cite[Proposition~7.1]{Miy87}, we obtain
\begin{align}
\begin{split}
\label{eq-Miyaoka-4}
2\widehat{c}_2(\Omega_{X}^{[1]})\polar
&\underalign{\text{(by Lem.~\ref{lem-c2-in-short-exact-seq})}}{\geq}
\left(2\sum_{i=1}^{l} \widehat{c}_2(\mathcal{G}_{i}^{\vee\vee})
-   c_1(\mathcal{G}_{i})^{2}\right)\polar
+  c_1(\Omega_{X}^{[1]})^{2}\polar \\
&\underalign{\text{(by \eqref{eq-Delta2-Higgs-G1})}}{\geq}
\left(- \frac{1}{r_1 + 1} c_1(\mathcal{G}_1)^2 -\sum_{i=2}^{l} \frac{1}{r_i} c_1(\mathcal{G}_{i})^{2}\right)\polar 
+ c_1(\Omega_{X}^{[1]})^2\polar\\
&\underalign{\text{(by Prop.~\ref{prop-hodge-index-type})}}{\geq}
- \frac{r_1^2}{d(r_1 + 1)}\mu_1^2 -\sum_{i=2}^{l} \frac{1}{r_i} c_1(\mathcal{G}_{i})^{2}\polar 
+ c_1(\Omega_{X}^{[1]})^2\polar \\
&\underalign{\text{(by \eqref{eq-Miyaoka-2})}}{\geq}
\frac{1}{d}
\left( \frac{r_1}{r_1 + 1} \mu_1^2 - d \mu_1 \right) 
+ c_1(\Omega_{X}^{[1]})^2\polar.
\end{split}
\end{align}
Consider the function
\[
f\colon \R \longrightarrow \R, \quad f(x) = \frac{r_1 x^2}{r_1 + 1} - d x.
\]
Note that $f$ is convex and attains its unique minimum at $x = \frac{(r_1 + 1)d}{2r_1}$. 
We observe that 
$$
\mu_1 \underalign{\text{(by \eqref{eq-minimal-variety-generically-nef})}}{\leq} \frac{d}{r_1} 
\underalign{}{<}
\frac{(r_1 + 1)d}{2r_1}.
$$ 
Thus, we obtain
\[
2\widehat{c}_2(\Omega_{X}^{[1]})\polar
\geq \frac{f\!\left(\frac{d}{r_1}\right)}{d}  + c_1(\Omega_{X}^{[1]})^2\polar
= \left( 
 \frac{1}{r_1(r_1 + 1)} - \frac{1}{r_1}  + 1 \right)c_1(\Omega_{X}^{[1]})^2 \polar,
\]
which yields the desired inequality \eqref{eq-Miyaoka-inequality-1}. 
Here, we used $d := c_1(\Omega_{X}^{[1]})^{2} H_1 \cdots H_{n-2}$ for the last equality.

We assume that $X$ is maximally quasi-\'etale and that equality holds in \eqref{eq-Miyaoka-inequality-1}.   
In the case $r_1 = n$, this situation has essentially been addressed in \cite{GKPT19a}. 
Specifically, Proposition \ref{prop-Higgs-2} shows that 
$\End(\Omega_{X}^{[1]} \oplus \mathcal{O}_{X})$ is locally free. 
Since $\Omega_{X}^{[1]}$ is a direct summand of $\End(\Omega_{X}^{[1]} \oplus \mathcal{O}_{X})$, 
the cotangent sheaf  $\Omega_{X}^{[1]}$ is also locally free. 
Therefore, by \cite[Theorem~6.1]{GKKP11}, we conclude that $X$ is smooth.

We now consider the case of $r_1 < n$. 
Since equality  holds in \eqref{eq-minimal-variety-generically-nef}, we have $r_1 \mu_1 = n\mu$. 
It follows that $\mu_2 = 0$, and by the definition of the Harder-Narasimhan filtration, we conclude that $l = 2$. 
By the same argument as in Theorem~\ref{thm-Langer-Liu-Liu-inequality}, we obtain $c_1(\mathcal{G}_2) = 0$.

Our next step is to prove that $\mathcal{G}_1$ and $\mathcal{G}_{2}^{\vee\vee}$ are locally free.
Since equality holds in \eqref{eq-Delta2-Higgs-G1}, 
we deduce that 
\begin{equation*} 
\label{eq-Miyaoka-Uniformisation}
\bigl( 2(r_1 + 1)\,\widehat{c}_2(\mathcal{G}_1) - r_1\, c_{1}(\mathcal{G}_1)^{2} \bigr)\polar = 0.
\end{equation*}
Thus, Corollary \ref{cor-locfree} implies that $\mathcal{G}_1$ is locally free.
Since $\mathcal{G}_{2}^{\vee\vee}$ is generically nef and satisfies $c_1(\mathcal{G}_{2}) = 0$, 
it follows that $\mathcal{G}_{2}^{\vee\vee}$ is $(H_1 \cdot H_2 \cdots H_{n-2})$-semistable.
Moreover, we obtain 
$$
\widehat{c}_2(\mathcal{G}_{2}^{\vee\vee})H_1 \cdots  H_{n-2}=0.
$$
Indeed, in \eqref{eq-Miyaoka-4} we applied the Bogomolov-Gieseker inequality to $\mathcal{G}_2$, 
and equality now holds. Thus, the above equality follows from $c_1(\mathcal{G}_2) = 0$.

Thus \cite[Theorem 1.4]{LT18} shows that $\mathcal{G}_{2}^{\vee\vee}$ is a flat locally free sheaf. 
Moreover, the torsion-free sheaf $\mathcal{G}_{2}$ is locally free in codimension two 
since $\mathcal{G}_{2} \to \mathcal{G}_{2}^{\vee\vee}$ is  isomorphic in codimension two 
by Lemma \ref{lem-c2-in-short-exact-seq}.

We finally prove that $\mathcal{G}_2 = \mathcal{G}_{2}^{\vee\vee}$ and that $X$ is smooth. 
Consider the short exact sequence
\begin{equation*}
0 \longrightarrow \mathcal{G}_1 \longrightarrow \Omega_{X}^{[1]} \longrightarrow \mathcal{G}_{2} \longrightarrow 0.
\end{equation*}
Note that $X$ satisfies Serre's condition $S_k$ for all $k$, since $X$ has klt singularities 
(and is therefore Cohen-Macaulay). 
Hence all the assumptions of \cite[Lemma~9.9]{AD14} are satisfied, and we deduce that 
$\Omega_{X}^{[1]}$ is locally free and $\mathcal{G}_2 = \mathcal{G}_{2}^{\vee\vee}$. 
In particular, it follows that $X$ is smooth by applying \cite[Theorem~6.1]{GKKP11} once more.
\end{proof}

\begin{thm} [{cf.\,\cite[Chapter 7]{Miy87}}]
\label{thm-Miyaoka-inequality-2}
Consider Setup \ref{setup-minimal-variety} and assume that $r_1 =1$.
Then, the following inequality holds$:$
\begin{equation}
\label{eq-Miyaoka-inequality-2}
\widehat{c}_2(\Omega_{X}^{[1]})H_1\cdots H_{n-2} > \frac{3}{8} c_1(\Omega_{X}^{[1]})^{2} H_1\cdots H_{n-2}.
\end{equation}
\end{thm}
We emphasize that equality never occurs in \eqref{eq-Miyaoka-inequality-2}.

\begin{proof}
Set $\polar := H_{1} \cdots H_{n-2}$. 
Since $\mathcal{G}_1 \subset \Omega_{X}^{[1]}$ and $\mu_1 > 0$ holds, 
we obtain 
\begin{equation}
\label{eq-square-seminegative}
c_1(\mathcal{G}_1)^2 \polar \leq 0
\end{equation}
by Propositions \ref{prop-Higgs-1} and \ref{prop-Higgs-2}.

Hence, as in the proof of Theorem~\ref{thm-Miyaoka-inequality-1}, we obtain
\begin{align}
\begin{split}
\label{eq-Miyaoka-r1-1}
2\widehat{c}_2(\Omega_{X}^{[1]})\polar
&\underalign{\text{(by \eqref{eq-Miyaoka-4})}}{\ge} \quad
\left(- c_1(\mathcal{G}_1)^2 -\sum_{i=2}^{l} \frac{1}{r_i} c_1(\mathcal{G}_{i})^{2}\right)\polar
+ c_1(\Omega_{X}^{[1]})^2 \polar\\
&\underalign{\text{(by \eqref{eq-square-seminegative})}}{\ge} \quad
c_1(\Omega_{X}^{[1]})^2\polar -\sum_{i=2}^{l} \frac{1}{r_i} c_1(\mathcal{G}_{i})^{2}\polar \\
&\underalign{\text{(by \eqref{eq-Miyaoka-2})}}{\ge} \quad
d- \mu_1+ \frac{\mu_{1}^{2}}{d} \\
& \ \ \ \ge \frac{3}{4}d.
\end{split}
\end{align}
Here, for the last inequality, we used the fact that the function
$
f(x) := d - x + \frac{x^2}{d}
$
attains its minimum value $\frac{3}{4}d$ at $x = \frac{d}{2}$. 
Hence the desired inequality \eqref{eq-Miyaoka-inequality-2} follows.

We finally show that equality is never attained in \eqref{eq-Miyaoka-inequality-2}.  
If equality in \eqref{eq-Miyaoka-inequality-2} were to hold, 
then by \eqref{eq-Miyaoka-2} and \eqref{eq-Miyaoka-r1-1} we would obtain $l = 2$, 
\[
\mu_2 = 0, 
\quad \text{and} \quad 
\mu_1 = \frac{d}{2}.
\]
However, this contradicts the relation $\mu_1 + r_2 \mu_2 = d$.
\end{proof}

By combining Theorems \ref{thm-Miyaoka-inequality-1} and \ref{thm-Miyaoka-inequality-2}, 
we derive Miyaoka's inequality for minimal projective klt varieties. 
Moreover, we analyze the case of equality and obtain a uniformization theorem.

\begin{thm}
[{cf.\,\cite[Chapter 7]{Miy87}, \cite[Theorem 5.6]{Lan02}, \cite[Theorem 7.2]{RT16}}]
\label{thm-Miyaoka-equality-case}
Consider Setup \ref{setup-minimal-variety}. 
Then, Miyaoka's inequality holds$:$
\begin{equation}
\label{eq-Miyaoka-equality-case}
\left( 3\widehat{c}_2(\Omega_{X}^{[1]}) - c_{1}(\Omega_{X}^{[1]})^{2}\right)H_1 \cdots H_{n-2} \ge 0
\end{equation}

Moreover, if equality holds in \eqref{eq-Miyaoka-equality-case}, then there exists a finite quasi-\'etale cover $A \times S \rightarrow X$, where $A$ is an abelian variety and $S$ is a smooth projective surface whose universal cover is 
the unit ball in $\C^2$. In particular, the canonical divisor $K_X$ is semi-ample.
\end{thm}
\begin{proof}
We easily obtain the desired inequality \eqref{eq-Miyaoka-equality-case} 
from Theorems \ref{thm-Miyaoka-inequality-1} and \ref{thm-Miyaoka-inequality-2},
by noting that 
\begin{equation}
\label{eq-1/3-inequality}
\frac{x}{2(x+1)} \geq \frac{1}{3} 
\end{equation}
for any real number $x \geq 2$.

Assume that equality holds in \eqref{eq-Miyaoka-equality-case}. 
After taking a quasi-\'etale cover, we may assume that $X$ is maximally quasi-\'etale.
Let $r_1$ be the rank of the maximal destabilizing subsheaf of $\Omega_{X}^{[1]}$ 
with respect to $K_X \cdot H_{1} \cdots H_{n-2}$. 
Then $r_1 = 2$ by Theorems \ref{thm-Miyaoka-inequality-1}, \ref{thm-Miyaoka-inequality-2}, and \eqref{eq-1/3-inequality}.
Moreover, applying Theorem \ref{thm-Miyaoka-inequality-1} again,
we see that $X$ is smooth and that there exists an exact sequence of locally free sheaves
\[
0 \longrightarrow \mathcal{G}_1 \longrightarrow \Omega_{X}^{1} \longrightarrow \mathcal{G}_2 \longrightarrow 0,
\]
where $\mathcal{G}_{2}$ is flat.
Since $\mathcal{G}_2^{\vee} \subset T_X$ has trivial first Chern class,  
the sheaf $\mathcal{G}_2^{\vee}$ defines a flat foliation of codimension two  
by \cite[Theorem~5.2]{LPT18}. 
Thus, by \cite[Theorem~D]{PT13}, we conclude that 
$X \cong A \times S$ up to a finite \'etale cover, 
where $A$ is an abelian variety and $S$ is a point, a smooth curve, or a smooth surface. 
By noting $\nu(K_X) = \nu(K_S) \geq 2$, we see that $S$ is a smooth projective surface of general type. 
Finally, equality in \eqref{eq-Miyaoka-equality-case} implies that $3c_2(S) - c_1(S)^2 = 0$, 
which shows that the universal cover of $S$ is the unit ball in $\C^2$ by \cite[Proposition~2.1.1]{MR0744605}.
\end{proof}

By Theorem \ref{thm-main1}, we immediately obtain Corollary \ref{cor-Abundance-theorem-c2}.

\begin{proof}[Proof of Corollary \ref{cor-Abundance-theorem-c2}]
By Theorem \ref{thm-main1}(A), we have 
$c_1(\Omega_{X}^{[1]})^2 H_1 \cdots H_{n-2} = 0$. 
Hence all the assumptions of Theorem \ref{thm-main1}(B) are satisfied, 
and we conclude that $K_X$ is semi-ample.
\end{proof}

\section{Proof of Theorem \ref{thm-main2}}
\label{Sec-proof3}
This section is devoted to the proof of Theorem \ref{thm-main2}. 
We begin by recalling the definition of the Albanese map for a projective klt variety $X$. 
Given a resolution $\pi \colon \widetilde{X} \to X$ of singularities, 
the Albanese map $\widetilde{\alpha} \colon \widetilde{X} \to \Alb(\widetilde{X})$ 
factors through $X$ (see, for example, \cite[Proposition~9.12]{Uen75}):
\[
\xymatrix{
\widetilde{X}\ar[r]^{\pi}\ar[dr]_{\widetilde{\alpha}} & X\ar[d]^{\alpha} \\
& \Alb(\widetilde{X}). 
}
\]
The morphism $\alpha \colon X \to \Alb(X) \coloneqq \Alb(\widetilde{X})$ 
is independent of the choice of resolution of singularities and is called the \textit{Albanese map} of $X$. 
Note that $\dim \Alb(X) = q(\widetilde{X}) = q(X)$ since $X$ has rational singularities, 
where $q(X) \coloneqq h^1(X, \mathcal{O}_X)$ denotes the irregularity of $X$. 
The \textit{augmented irregularity} $\widehat{q}(X)$ of $X$ is defined as
\[
\widehat{q}(X) \coloneqq \sup \{ q(\widehat{X}) \,\mid\, \widehat{X} \to X \text{ is a finite quasi-\'etale cover} \} 
\in \mathbb{Z}_{\geq 0} \cup \{\infty\}.
\]
The following theorem is an immediate consequence of \cite{ MW21} 
and the Beauville-Bogomolov-Yau decomposition theorem for klt varieties 
(see \cite{GKP16b, Dru18, GGK19, HP19, Cam20}). 

\begin{thm}
\label{thm-fibre-zero-aug}
Let $X$ be a projective klt variety with nef anti-canonical divisor. 
Then, after replacing $X$ with a finite quasi-\'etale cover, there exist the following data:
\begin{enumerate}[label=$(\arabic*)$]
\item an abelian variety $A$ of dimension $q := \widehat{q}(X)$,
\item a projective klt variety $Y$ with $K_Y = 0$ and $\widehat{q}(Y) = 0$,
\item a rationally connected klt variety $F$ with $-K_F$ nef,
\item a group homomorphism $\rho \colon \pi_1(A) \to \mathrm{Aut}^0(F)$
\end{enumerate}
such that
\[
X \cong \Big( (\C^q \times F)/\pi_1(A) \Big) \times Y,
\]
where $\pi_1(A)$ acts diagonally on $\C^q \times F$ via 
$\rho \colon \pi_1(A) \to \mathrm{Aut}^0(F)$.
In particular, the Albanese map $\alpha \colon X \to A$ is a locally trivial fibration
whose fibres are isomorphic to $F \times Y$.
\end{thm}

\begin{proof}
By combining \cite{MW21} with the Beauville-Bogomolov-Yau decomposition theorem for klt varieties, 
and after replacing $X$ with a finite quasi-\'etale cover, 
we obtain an abelian variety $A$ of dimension $q := \widehat{q}(X)$, 
a projective klt variety $Y$ with $K_Y = 0$ and $\widehat{q}(Y) = 0$, 
a rationally connected klt variety $F$ with nef anti-canonical divisor, 
and a group homomorphism $\rho \colon \pi_1(A) \to \Aut{F}$ such that
\[
X \cong \Big( (\C^q \times F)/\pi_1(A) \Big) \times Y.
\]
Here, we used the fact that every linear representation of $\pi_1(Y)$ has finite image
(see \cite[Theorem~I]{GGK19}).
Since $X$ is projective, the image of the composition 
\[
\pi_1(A) \longrightarrow \Aut{F} \longrightarrow \Aut{F}/\mathrm{Aut}^0(F)
\] 
is finite (see, for example, \cite[Lemma~3.4]{M22}). 
Hence, after replacing $X$ with a finite \'etale cover, 
we may assume that $\rho$ takes values in $\mathrm{Aut}^0(F)$.  
\end{proof}

After proving that the tangent sheaf $\mathcal{T}_X$ is generically nef, 
we now complete the proof of Theorem~\ref{thm-main2}. 

\begin{lem}
\label{lem-generic-semipos-tangent-antican-nef}
Let $X$ be a projective klt variety with nef anti-canonical divisor.
Then, the tangent sheaf $\mathcal{T}_X$ is generically nef.
In particular, for any ample Cartier divisors $H_1, \ldots, H_{n-2}$ on $X$, we have
$$
\widehat{c}_2(\mathcal{T}_X) \cdot H_1 \cdots H_{n-2} \ge 0.
$$
\end{lem}

\begin{proof}
Note that when $X$ is $\Q$-factorial, the conclusion was proved in \cite[Theorem~1.4]{Ou17}.
In the general case, given a subsheaf $\mathcal{E} \subset \Omega_X^{[1]}$, 
we show that
\[
c_1(\mathcal{E}) \cdot H_1 \cdots H_{n-1} \leq 0
\]
for any ample divisors $H_1, \ldots, H_{n-1}$ on $X$.
Let $\rho \colon \widetilde{X} \to X$ be a $\Q$-factorization as in \cite[Corollary~1.37]{Kol13},
and set $\widetilde{\mathcal{E}} := \rho^{[*]}\mathcal{E}$.
Since $\rho$ is a small birational morphism
(i.e.,\,it does not contract any divisor), we have
\[
K_{\widetilde{X}} = \rho^*K_X
\quad \text{and} \quad
c_1(\widetilde{\mathcal{E}}) = \rho^*c_1(\mathcal{E}).
\]
We see that $-K_{\widetilde{X}}$ is also nef.
Moreover, we have $\Omega_{\widetilde{X}}^{[1]} \cong \rho^{[*]}\Omega_X^{[1]}$ by reflexivity,
and thus $\widetilde{\mathcal{E}} $ can be regarded as a subsheaf  of $\Omega_{\widetilde{X}}^{[1]}$.
By \cite[Theorem~1.4]{Ou17}, we deduce that
\[
c_1(\mathcal{E}) \cdot H_1 \cdots H_{n-1}
= c_1(\widetilde{\mathcal{E}}) \cdot \rho^*H_1 \cdots \rho^*H_{n-1}
\leq 0.
\]
This shows that $\mathcal{T}_X$ is generically nef.
The latter conclusion follows from Corollary~\ref{cor-inequality-c2-nu2} and Theorem~\ref{thm-c2-inequality-nu1}.
\end{proof}

\begin{proof}[Proof of Theorem~\ref{thm-main2}]
Let $X$ be a projective klt variety of dimension $n$ satisfying the assumptions of Theorem~\ref{thm-main2}. 
Note that, by Proposition~\ref{prop-vanishing-intersection-form} and Lemma~\ref{lem-generic-semipos-tangent-antican-nef}, 
we have $\widehat{c}_2(\mathcal{T}_{X}) \cdot D_1 \cdots D_{n-2} = 0$ for all nef $\Q$-divisors $D_1, \ldots, D_{n-2}$.

It suffices to prove the conclusion after replacing $X$ with a finite quasi-\'etale cover. 
Hence, we may assume that $X$ is maximally quasi-\'etale. 
Furthermore, by Theorem~\ref{thm-fibre-zero-aug}, 
we may assume that 
\[
X \cong \big((\C^q \times F)/\pi_1(A)\big) \times Y
\]
and that $X$ satisfies the properties listed in Theorem~\ref{thm-fibre-zero-aug}. 
Recall that $Y$ is a projective klt variety with $K_Y = 0$ and $\widehat{q}(Y) = 0$, 
and that $F$ is a rationally connected klt variety.
Then, from the assumption on $\widehat{c}_2(\mathcal{T}_{X}) $, 
we obtain the following claim.

\begin{claim}
\label{claim-single}
The variety $Y$ is a single point.
\end{claim}

\begin{proof}[Proof of Claim \ref{claim-single}]
Assume for a contradiction that $Y$ is not a point.
Consider the splitting of the tangent sheaf 
$$
\mathcal{T}_X = \pr_{1}^{*}\mathcal{T}_{Z} \oplus \pr_{2}^{*}\mathcal{T}_{Y}
$$
arising from the product structure $X \cong Z \times Y$, 
where 
\[
Z := (\C^q \times F)/\pi_1(A)
\]
and $\pr_{i}$ is the natural projection. 
Let $D$ be an ample divisor on $Z$ and let $H$ be an ample divisor on $X$. 
Set $m := \dim Z$. Then, Proposition~\ref{prop-vanishing-intersection-form} shows that
\[
\widehat{c}_2(\mathcal{T}_{X}|_Y) \cdot (H|_Y)^{n-m-2}
= \widehat{c}_2(\mathcal{T}_{X}) \cdot H^{n-m-2} \cdot (\pr_{1}^*D)^m = 0.
\]
By combining Lemma~\ref{lem-c2-in-short-exact-seq} and Lemma~\ref{lem-generic-semipos-tangent-antican-nef}, 
we deduce that
\[
\widehat{c}_2(\mathcal{T}_{Y}) \cdot (H|_Y)^{n-m-2} = 0.
\]
Since $K_Y = 0$, it follows from \cite[Theorem~1.2]{LT18} that 
$Y$ is a quasi-\'etale quotient of an abelian variety, contradicting $\widehat{q}(Y) = 0$. 
\end{proof}

From now on, let $\alpha \colon X \to A$ be the Albanese morphism of $X$ as in Theorem~\ref{thm-fibre-zero-aug}. 
By construction and Claim~\ref{claim-single}, this is a locally trivial fibration with fibre $F$, where $F$ is a rationally connected klt variety. 
By \cite[Theorem~1.5]{EIM21} (cf.~\cite[Theorem~1.2]{CCM21}), we have 
\begin{align}\label{eq-HM}
\nu(-K_{X}) = \nu(-K_{F}) \leq \dim F. 
\end{align}
Thus, to conclude the proof of Theorem~\ref{thm-main2}, it remains to show that $\dim F \leq 1$. 
To this end, we consider three cases according to the numerical dimension $\nu(-K_X)$ of $-K_X$. 
We first treat the case $\nu(-K_X) = 0$.

\begin{case}[The case $\nu(-K_X) = 0$]
In this case, by \cite[Theorem~1.2]{LT18}, the variety $X$ is a quasi-\'etale quotient of an abelian variety, 
and thus satisfies the desired conclusion. 
\end{case} 

In the case $\nu(-K_X) \geq 1$, we consider the relative tangent sheaf $\mathcal{T}_{X/A}$. 
Note that we have the splitting 
$$
\mathcal{T}_X = \mathcal{T}_{X/A} \oplus \alpha^{*}\mathcal{T}_{A}
\cong \mathcal{T}_{X/A} \oplus \mathcal{O}_{X}^{\oplus q},
$$
since $\alpha \colon X \to A$ is a locally constant fibration. 

\begin{case}[The case $\nu(-K_X) \geq 2$]
We show that this situation cannot occur.
In this case, we have $\dim F \geq 2$ by \eqref{eq-HM}. 
The relative tangent sheaf $\mathcal{T}_{X/A}$ satisfies the assumptions of Corollary~\ref{cor-inequality-c2-nu2}, 
and hence, we obtain a surjective morphism of sheaves $\mathcal{T}_{X/A} \twoheadrightarrow \mathcal{Q}$, 
where $\mathcal{Q}$ is a flat locally free sheaf such that $\rk \mathcal{Q} = \dim F - 1\geq 1$. 
Taking dual and restricting to the fibre, 
we obtain a flat subsheaf 
$$
(\mathcal{Q}|_{F}) ^{\vee}\hookrightarrow (\mathcal{T}_{X/A}|_{F})^{\vee} \cong \Omega_F^{[1]}. 
$$
The flat locally free sheaf $(\mathcal{Q}|_{F}) ^{\vee}$ is trivial 
since $F$ is rationally connected (in particular, simply connected). 
This implies that $q(F) \neq 0$, contradicting the assumption that $F$ is rationally connected.
\end{case}

\begin{case}[The case $\nu(-K_X)=1$]
Recall that it remains only to show that $\dim F \leq 1$. 
To this end, assume $\dim F \geq 2$ and derive a contradiction.

Fix an ample divisor $H$ on $X$ and set $H_{\varepsilon} := -K_X + \varepsilon H$. 
Consider the $(H_{\varepsilon} \cdot H_1 \cdots H_{n-2})$-Harder-Narasimhan filtration of $\mathcal{T}_{X/A}$: 
\begin{align}\label{fil-1}
0 =: \mathcal{E}_0 \subsetneq \mathcal{E}_1 \subsetneq \cdots \subsetneq \mathcal{E}_l := \mathcal{T}_{X/A}.
\end{align}
By Theorem~\ref{thm-c2-inequality-nu1}, this filtration is independent of $\varepsilon$ for $0 < \varepsilon \ll 1$. 
Moreover, the graded pieces $\mathcal{G}_i := \mathcal{E}_i/\mathcal{E}_{i-1}$ 
satisfy the following properties:
\begin{enumerate}
\item $\mathcal{G}_i$ is a reflexive sheaf of rank $r_i$ on $X$.
\item $\mathcal{G}_i|_{X_{\reg}}$ is locally free and projectively flat, with associated representation $\rho_i \colon \pi_1(X_{\reg}) \to \mathbb{P}\GL(r_{i},\C)$.
\item $\det(\mathcal{G}_i)$ is a $\Q$-line bundle, and 
$c_1(\mathcal{G}_i) = \lambda_i c_{1}(\mathcal{T}_{X/A}) = \lambda_i c_{1}(-K_{X})$. 
\end{enumerate}
Then, by applying \cite[Proposition~3.3]{GKP21} to each $\mathcal{G}_i$, 
there exists a reflexive sheaf $\mathcal{L}_{i}$ of rank one 
such that $\mathcal{G}_i = \mathcal{F}_{i} \otimes \mathcal{L}_{i}$, 
where $\mathcal{F}_{i}$ admits a filtration 
\begin{align}\label{fil-2}
0 =: \mathcal{F}_i^{0} \subsetneq \mathcal{F}_i^{1} \subsetneq \cdots \subsetneq \mathcal{F}_i^{r_{i}-1} \subsetneq \mathcal{F}_i^{r_{i}} := \mathcal{F}_i 
\end{align}
such that each graded piece $\mathcal{F}_i^{k}/\mathcal{F}_i^{k-1}$ is a flat line bundle, 
and $\mathcal{F}_i^{1}$ is a trivial line bundle on $X$. 

Now, by a straightforward computation, 
we have 
$$
\lambda_{i} c_{1}(-K_{X})
= c_{1}(\mathcal{G}_i)
=r_{i} c_{1}(\mathcal{L}_{i}).
$$
Then, from the properties of the Harder-Narasimhan filtration \eqref{fil-1} and Lemma~\ref{lem-generic-semipos-tangent-antican-nef}, we deduce that
\begin{equation}
\label{eq-HNfiltration-nu11}
\frac{\lambda_1}{r_{1}} > \frac{\lambda_2}{r_{2}} > \cdots > \frac{\lambda_l}{r_{l}} \;
\underalign{\text{(by Lem.~\ref{lem-generic-semipos-tangent-antican-nef})}}{\geq}
0
\quad \text{and} \quad \sum_{i=1}^{l} \lambda_{i} \underalign{\text{(by \eqref{fil-1})}}{=} 1.
\end{equation}

We now show that $\lambda_{l} > 0$.
Suppose instead that $\lambda_{l} = 0$. 
Then, the $\Q$-line bundle $\mathcal{L}_l$ is numerically trivial. 
On the other hand, by definition, there is a surjective morphism of sheaves
$$
\mathcal{T}_{X/A} = \mathcal{E}_{l} \twoheadrightarrow
\mathcal{E}_{l}/\mathcal{E}_{l-1} = \mathcal{G}_{l} 
= \mathcal{F}_{l} \otimes \mathcal{L}_{l}
\twoheadrightarrow \mathcal{F}_l^{r_{l}}/\mathcal{F}_l^{r_{l}-1} \otimes \mathcal{L}_{l} =: \mathcal{M}.
$$
Since $X$ is maximally quasi-\'etale and $c_1(\mathcal{M})=0$, the sheaf $\mathcal{M}$ is a flat line bundle by \cite[Theorem~1.4]{LT18} (see also the proof of Lemma~\ref{lem-Q-Cartier-of-Gi}).
By taking  dual and restricting to the fibre, 
we obtain the inclusion 
$$
(\mathcal{M}|_{F})^{\vee} \hookrightarrow (\mathcal{T}_{X/A}|_{F})^{\vee} \cong \Omega_F^{[1]}. 
$$
The flat line bundle $(\mathcal{M}|_{F})^{\vee}$ is trivial 
since $F$ is rationally connected. 
This implies that $q(F) \neq 0$, which contradicts the fact that $F$ is rationally connected.

Consequently, we obtain $\lambda_l > 0$. 
Hence, by \eqref{eq-HNfiltration-nu11} and $\dim F \geq 2$, 
we deduce that $\lambda_i / r_i < 1$ for every $i$. 
This establishes Claim~\ref{claim-extremal-nefness} below, 
which implies that $K_F$ is nef and therefore numerically trivial. 
This contradicts \eqref{eq-HM}, and thus we conclude that $\dim F \leq 1$.

\begin{claim}
\label{claim-extremal-nefness}
$F$ has no $K_F$-negative extremal contraction. In particular, $K_F$ is nef.
\end{claim}

\begin{proof}[Proof of Claim \ref{claim-extremal-nefness}]
Assume that $F$ admits a $K_F$-negative extremal contraction $\phi \colon F \to W$. 
By the argument in \cite[Claim~5.8 in the ArXiv version]{GKP21}, it suffices to show that 
\[
R^p \phi_{*}\Omega_{F}^{[1]} = 0 \quad \text{for all } p > 0. 
\]
To this end, restricting \eqref{fil-1} and \eqref{fil-2} to the fibre $F$ and taking duals, 
it is enough to verify that
\begin{equation}\label{eq-vanishing}
R^p \phi_{*}
\Big( 
\big(\mathcal{F}_{i}^{k}/\mathcal{F}_{i}^{k-1}\big)^{\vee} \otimes \mathcal{L}_{i}^{\vee}
\Big) 
= 0 \quad \text{for all } p > 0. 
\end{equation}
Note that 
\[
\mathcal{O}_F(-K_{F}) \otimes
\big(\mathcal{F}_{i}^{k}/\mathcal{F}_{i}^{k-1}\big)^{\vee} \otimes \mathcal{L}_{i}^{\vee} 
\]
is numerically equivalent to $(1 - \lambda_{i}/r_{i}) c_{1}(-K_{F})$, 
and is therefore $\phi$-ample since $\lambda_i / r_i < 1$. 
Thus, by the relative vanishing theorem 
(see \cite[Theorem~1-2-5]{KMM87}), 
we obtain the desired vanishing \eqref{eq-vanishing}.
\end{proof}

\end{case}
In summary, we see that in all cases $\dim F \leq 1$, as required.
This completes the proof.  
\end{proof}

\bibliographystyle{alpha}
\bibliography{ref_minimal.bib}
\end{document}